\documentclass{siamart0516} 

\usepackage{amsfonts,amssymb,mathtools,bm}
\usepackage[OT1]{fontenc}
\usepackage[inline]{enumitem}
\usepackage{color,comment,appendix}
\usepackage{caption,subcaption}

\newtheorem{example}{Example}
\newtheorem{remark}{Remark}

\def\eqns#1{\begin{equation*}#1\end{equation*}}
\def\eqnl#1#2{\begin{equation}\label{#1}#2\end{equation}}
\def\eqnsml#1{\begin{multline*}#1\end{multline*}}
\def\eqnlml#1#2{\begin{multline}\label{#1}#2\end{multline}}
\def\eqnsa#1{\begin{subequations}\begin{align*}#1\end{align*}\end{subequations}}
\def\eqnmla#1#2{\begin{subequations}\label{#1}\begin{align}#2\end{align}\end{subequations}}

\def\Bi{\mathrm{Bi}}
\def\Po{\mathrm{Po}}
\def\C{\mathrm{C}}
\def\D{\mathrm{D}}
\def\d{\mathrm{d}}
\def\m{\mathrm{m}}
\def\id{\mathrm{id}}
\def\loss{\mathrm{loss}}
\def\T{\mathrm{T}}

\def\one{\mathbf{1}}
\def\bsb{\bm{b}}
\def\bsd{\bm{d}}
\def\bsD{\bm{D}}
\def\bse{\bm{e}}
\def\bsf{\bm{f}}
\def\bsg{\bm{g}}
\def\bsh{\bm{h}}
\def\bsI{\bm{I}}

\def\bsp{\bm{p}}
\def\bsP{\bm{P}}
\def\bsq{\bm{q}}

\def\bsx{\bm{x}}
\def\bsX{\bm{X}}
\def\bsy{\bm{y}}
\def\bsY{\bm{Y}}
\def\bseta{\bm{\eta}}

\def\bsPi{\bm{\Pi}}
\def\bsrho{\bm{\rho}}
\def\bstheta{\bm{\theta}}
\def\bsTheta{\bm{\Theta}}

\def\calB{\mathcal{B}}

\def\calF{\mathcal{F}}
\def\calN{\mathcal{N}}
\def\calX{\mathcal{X}}

\def\bbE{\mathbb{E}}
\def\bbN{\mathbb{N}}
\def\bbP{\mathbb{P}}
\def\bbR{\mathbb{R}}
\def\bbX{\mathbb{X}}
\def\bbY{\mathbb{Y}}

\DeclareMathOperator{\Sym}{Sym}
\DeclareMathOperator{\Sco}{Sco}

\def\AND{\qquad\mbox{ and }\qquad}
\def\And{,\;}
\def\defeq{\doteq}
\def\given{\,|\,}
\def\Given{\,\big|\,}
\def\ind#1{\one_{#1}}
\def\st{:}
\def\tr{\mathrm{t}}

\title{Identification of multi-object dynamical systems: consistency and Fisher information}%

\author{Jeremie Houssineau%
\thanks{DSAP, National University of Singapore. Email: \href{mailto:stahje@nus.edu.sg}{stahje@nus.edu.sg}}
\and
Sumeetpal S.\ Singh%
\thanks{Department of Engineering, University of Cambridge  and The Alan Turing Institute. Email:~\href{mailto:sss40@cam.ac.uk}{sss40@cam.ac.uk}}
\and
Ajay Jasra%
\thanks{DSAP, National University of Singapore. Email: \href{mailto:staja@nus.edu.sg}{staja@nus.edu.sg}}
}

\headers{Identification of multi-object dynamical systems}
{J.\ Houssineau, S.S.\ Singh and A.\ Jasra}

\begin{document}

\maketitle

\begin{abstract}
Learning the model parameters of a multi-object dynamical system from partial and perturbed observations is a challenging task. Despite recent numerical advancements in learning these parameters, theoretical guarantees are extremely scarce. In this article, we study the identifiability of these parameters and the consistency of the corresponding maximum likelihood estimate (MLE) under assumptions on the different components of the underlying multi-object system. In order to understand the impact of the various sources of observation noise on the ability to learn the model parameters, we study the asymptotic variance of the MLE through the associated Fisher information matrix. For example, we show that specific aspects of the multi-target tracking (MTT) problem such as detection failures and unknown data association lead to a loss of information which is quantified in special cases of interest.
\end{abstract}



\section{Introduction}

A multi-object dynamical system is comprised of an unknown and randomly varying number of objects, each of which is a partially observed Markov process. 
Multi-target tracking refers to the problem of estimating the state of each of these objects from noisy observations that are also corrupted by detection failures and false detections (a.k.a.\ false alarms). This type of problem arises in many different fields such as Systems Biology \cite{Chenouard2014}, Robotics \cite{Mullane2011}, Computer Vision \cite{Okuma2004} or Surveillance \cite{Pailhas2016}. Different formulations of multi-target tracking exist, including extensions of the single-target approach to multiple targets \cite{Blackman1986} as well as formulations based on simple point processes~\cite{Mahler2003}.

One of the main challenges in multi-target tracking is the uncertainty in the \emph{data association}, which refers to the problem of finding the right pairing between targets and recorded observations over time, a task further confounded by the corruption of these observations with false positives and detection failures. Inferentially, multi-target tracking is notoriously difficult to solve as it involves an exponentially growing numbers of possible configurations for the data association. Over the past decade there has been significant advancements towards more practical solutions to this inference problem. Some of these include solutions based on sequential Monte Carlo (SMC) \cite{Vo2005}, hierarchical SMC \cite{Pace2013} or Gaussian mixtures \cite{Vo2006}. 

In this article, both the MTT observation model and the motion model of the constituent individual targets are assumed unknown and are instead parameterised and to be inferred from the data. Although multi-target tracking has been an active research field for decades, questions concerning the identifiability and the consistency of the corresponding model parameter estimates have not received the appropriate attention. In this paper we aim to address this gap and shed some light on this issue. Building on results from the literature on Markov processes (e.g.\ see \cite{Leroux1992,Douc2004}), we prove both identifiability and the consistency of the MLE of the MTT model parameters in \cref{thm:consistencyMLE}. Specifically, as each constituent target of the MTT model is a partially observed Markov process, in \cref{thm:transferIdentifiability} we show that identifiability transfers from single to multiple targets under appropriate assumptions. The practical implications of results regarding identifiability include the understanding of the behaviour of Markov chain Monte Carlo (MCMC) techniques in multi-target tracking \cite{Oh2009,Jiang2015}, which is conditioned by the likelihood ratio between the correct parameter value and all the other possible values. The consistency of the maximum likelihood estimator raises the question of its asymptotic normality and the corresponding variance, which in turns motivates the study of the Fisher information matrix for this class of problems. It is demonstrated in \cref{thm:informationLoss} that there is a strict loss of information in the presence of data association uncertainty or detection failures. We characterise the Fisher information more precisely in specific illustrated cases, e.g.\  we show that when increasing the number of targets there is no gain in the Fisher information for the model parameters which are common to all targets if large uncertainties on the origin of the corresponding observations persist (see \cref{ssec:informationDataAssociation}). The Fisher information matrix is useful in applications such as sensor management \cite{Doucet2002} which aims at optimising the position of the sensor or at finding the best ratio between probability of false alarm and probability of detection.

The proof of identifiability of the MTT model as well as our approach for studying the asymptotic variance of the MLE for the MTT model parameters are original and, to the best of our knowledge, the first of their kind. Consistency of the data association problem in MTT has been studied in \cite{Storlie2011} in the context of the estimation of multiple splitting and merging targets observed without noise over a fixed time interval during which $n$ observations of the multiple targets are made at discrete times. The result in \cite{Storlie2011} is limited to the case where the number $n$ of observation tends to infinity which effectively amounts to saying that targets are observed infinitely many times over a fixed interval which is a scenario not typically encountered in practice. In any case, our theoretical results and proof techniques are entirely different as they pertain to the MTT model parameters and not the data association. Point-process-based theoretical studies of MTT have also been conducted in \cite{DelMoral2013,DelMoral2015} for the stability of specific inference methods.

The structure of the article is as follows: after introducing the required notations and background concepts in \cref{sec:notations} and \cref{sec:background}, the consistency of the maximum likelihood estimator is established along with its asymptotic normality for a large class of multi-object systems in \cref{sec:consistency}. Finally, in order to better understand the effect of the various parameters on the asymptotic variance, the Fisher information matrix is computed for important special cases of multi-object systems in \cref{sec:FisherInformation}. The article concludes in \cref{sec:conclusion}.

\section{Notations}
\label{sec:notations}

All random variables will be defined on the same probability space $(\Omega,\calF,\bbP)$ and the expectation of a random variable $X$ w.r.t.\ the probability measure $\bbP$ is denoted $\bbE[X]$. Probability densities will be denoted by lower-case letters while probability measures will be denoted by capital letter. Similarly, random variables will be denoted with capital letters whereas their realisations will be in lower-case.

The time is indexed by the set $\bbN$ of positive integers and for every time $t \in \bbN$, a finite sequence $\bsy_t$ of $M_t \in \bbN_0 \defeq \bbN \cup \{0\}$ observation points in the observation space $\bbY$ is made available. This space can be assumed to be a subset of the Euclidean space $\bbR^d$ with $d>0$. The sequences of observations of the form $(\bsy_1,\dots,\bsy_n)$ will be denoted $\bsy_{1:n}$. In the standard formulation of multi-target tracking, no more than one observation is associated with a given object at a given time step and, conversely, observations are originated from one object only. 

Objects' states are modelled as elements of a set $\bbX$ which is assumed to be a subset of the Euclidean space $\bbR^{d'}$ with $d' > 0$; usually satisfying $d' \geq d$. They are propagated independently according to a Markov kernel density $f_{\theta}$ from the state space $\bbX$ to itself, which depends on a parameter $\theta$ from a compact set $\Theta$. Densities on $\bbX$ are defined w.r.t.\ a reference measure $\mu$. The true value of the parameter $\theta$ is denoted $\theta^*$. The random variable $X_t$ describing the state at time $t$ only depends on the state $x_{t-1}$ at time $t-1$ as follows
\eqns{
X_t \sim f_{\theta}(\cdot \given x_{t-1}).
}
This transition does not depend on time so that the associated Markov chain is said to be \emph{homogeneous}. The observation process at time $t$ given the state $x_t$ is modelled by
\eqns{
Y_t \sim g_{\theta}(\cdot \given x_t)
}
where $g_{\theta}$ is a likelihood function from $\bbX$ to $\bbY$, also parametrised by $\theta$, so that the observation $Y_t$ at time $t$ is independent from the states and observations at other times. The process $(X_t,Y_t)_{t \geq 1}$ 
is usually referred to as a hidden Markov model (HMM). Its law under the parameter $\theta \in \Theta$ is denoted $\bar{P}_{\theta}$ when initialised with its stationary distribution assuming it exists, and $P_{\theta}(\cdot \given x_0)$ when initialised at $x_0 \in \bbX$.

\section{Background}
\label{sec:background}

The definition of specific properties of Markov chains that will be used in the following sections is given here for completeness. Let $(X_t)_{t \geq 0}$ be a $\bbX$-valued Markov chain with transition density $f$ and let $P(\cdot \given x)$ be the probability measure on $(\bbX^{\bbN_0},\calX^{\otimes \bbN_0})$, where $\calX^{\otimes \bbN_0}$ is the cylinder $\sigma$-algebra on $\bbX^{\bbN_0}$, characterising the chain when initialised at point $x \in \bbX$.  Also, let $\tau_A$ be the return time to a set $A \subseteq \bbX$ defined as $\tau_A = \inf\{ t \geq 1 : X_t \in A \}$.

Consider the following concepts: A set $A \subseteq \bbX$ is said to be \emph{accessible} if $\tau_A < \infty$ has positive probability under $P(\cdot \given x)$ for all $x \in \bbX$. The Markov chain $(X_t)_{t \geq 0}$ is said to be \emph{phi-irreducible} if there exists a density $\phi$ on $\bbX$ such that for any subset $A \subseteq \bbX$, $\int_A \phi(x) \d x > 0$ implies that $A$ is accessible. A set $A \subseteq \bbX$ is said to be \emph{Harris recurrent} if the event $\tau_A < \infty$ happens almost surely under $P(\cdot \given x)$ for all $x \in \bbX$. A phi-irreducible Markov chain is said to be Harris recurrent if any accessible set is Harris recurrent. A density $q$ is called \emph{invariant} if for all $x \in \bbX$ it holds that
\eqns{
q(x) = \int f(x \given x') q(x') \d x'.
}
A phi-irreducible Markov chain is called \emph{positive} if it admits an invariant probability density. More details about these notions expressed in a measure-theoretic formulation can be found in \cite{Meyn2012}. These concepts will be useful when considering the long-time behaviour of the Markov chains involved in multi-target tracking problems.

\section{Consistency of the maximum likelihood estimator}
\label{sec:consistency}

\subsection{The multi-target tracking model}

Throughout this section, the true number of objects in the considered system will be assumed to be fixed and will be denoted by $K^* \in \bbN$. We consider a Markov chain $(\bsX_t)_{t \geq 0}$ in $\bbX^{K^*}$ with components independently evolving via the Markov transition $f_{\theta}$ from $\bbX$ to $\bbX$. Observations at time $t$ are gathered into a vector $\bsy_t$ in the space $\bbY^{\times} \defeq \bigcup_{k \geq 0} \bbY^k$ where $\bbY^0$ is a notation for the set containing the empty sequence only. The observation $\bsy_t$ is a superposition of
\begin{enumerate}[wide]
\item the independent observation of components of $\bsX_t$ via the likelihood $g_{\theta}$ from $\bbX$ to $\bbY$ followed by a Bernoulli thinning with parameter $p_{\D}$ corresponding to detection failure, and
\item false alarms, or \emph{clutter}, generated independently of the object-originated observations and assumed to come from an i.i.d.\ process whose cardinality at each time is Poisson with parameter $\lambda$ and common distribution $P_{\psi}$ which depends on the parameter $\psi$ in a compact set $\Psi$ and which true value is denoted $\psi^*$.
\end{enumerate}
The number of objects $K^*$ is not assumed to be known so that it will also be considered as a parameter of the model. The parameter for the multi-target model is then defined as $\bstheta \defeq [\theta,K,p_{\D},\lambda,\psi]^{\tr} \in \bsTheta \defeq \Theta \times S^{\T} \times (0,1) \times S^{\C} \times \Psi$ where $\tr$ is the vector transposition and where $S^{\T}$ and $S^{\C}$ are compact subsets of $\bbN$ and $(0,\infty)$ respectively, with ``$\T$'' and ``$\C$'' standing for target and clutter respectively. The true parameter $\bstheta^*$ is assumed to be an interior point of $\bsTheta$. Special parameter sets that are not subsets of $\bsTheta$ can also be introduced by fixing one or several parameters to special values, for instance $\bsTheta_{\lambda=0} \defeq \Theta \times S^{\T} \times (0,1)$, $\bsTheta_{p_{\D}=1} \defeq \Theta \times S^{\T} \times S^{\C} \times \Psi$ or $\bsTheta_{\lambda = 0, p_{\D}=1} \defeq \Theta \times S^{\T}$ correspond respectively to cases where the parameters $\lambda$, $p_{\D}$ or both have known values that are outside of their domain of definition in $\bsTheta$. Alternatively, if the value of a parameter is known but inside of its domain of definition, e.g.\ it is known that $K=1$, then the corresponding hyperplane will be expressed as $\bsTheta|_{K=1}$. Although the Poisson distribution is not defined for the parameter $\lambda = 0$, this parameter value is simply assumed to represent the case where there is no false alarm.

The Markov transition $\bsf_{\bstheta}$ associated with the $K$-target process $(\bsX_t)_t$ can simply be expressed as
\eqns{
\bsf_{\bstheta}(\bsx \given \bsx') = \prod_{i = 1}^K f_{\theta}(\bsx_i \given \bsx'_i),
}
for any $\bsx, \bsx' \in \bbX^K$, the likelihood however takes a more sophisticated form so that additional notations are required. Let $\Sym(k)$ be the symmetric group over $k$ letters and $u_k$ be the uniform distribution over $\Sym(k)$, also let $\bsq_{\bstheta}$ be a distribution on $\{0,1\}^K$ characterised by
\eqns{
\bsq_{\bstheta}(\bsd) \defeq p_{\D}^{|\bsd|} (1-p_{\D})^{K - |\bsd|},
}
for any $\bsd \in \{0,1\}^K$. The variable $\bsd$ is such that $\bsd_i = 1$ if and only if target $i$ is detected for any $i \in \{1,\dots,K\}$. The $K$-target likelihood $\bsg_{\bstheta}(\bsy_t \given \bsx_t)$ of the observations $\bsy_t \in \bbY^{\times}$ at time $t$ given the state $\bsx \in \bbX^K$ is characterised by
\eqnlml{eq:targetLikelihood}{
\bsg_{\bstheta}(\bsy_t \given \bsx) \defeq \sum_{\substack{\bsd \in \{0,1\}^K \\ |\bsd| \leq M_t}} \bigg[ \Po_{\lambda}(M_t - |\bsd|) \\
\times \sum_{\sigma \in \Sym(M_t)} \prod_{i = |\bsd|+1}^{M_t} p_{\psi}\big(\bsy_{t,\sigma(i)}\big) \prod_{i = 1}^{|\bsd|} g_{\theta}\big(\bsy_{t,\sigma(i)} \given \bsx_{r(i)}\big) u_{M_t}(\sigma) \bsq_{\bstheta}(\bsd) \bigg],
}
where $\Po_{\lambda}$ denotes the Poisson distribution with parameter $\lambda$, where $|\bsd|$ is the 1-norm of $\bsd$, i.e.\ the number of detected targets and where $r(i)$ is the $i$\textsuperscript{th} detected target that is the integer verifying $|\bsd_{1:r(i)}| = i$. This choice of the likelihood $\bsg_{\bstheta}$ corresponds to a marginalisation over the observation-to-track data association. Note that $|\bsd| \leq K$ for any $\bsd \in \{0,1\}^K$ so that $\bsg_{\bstheta}(\bsy \given \bsx) = 0$ for any $\bsx \in \bbX^K$ if $\lambda = 0$ and if the number of observations in $\bsy$, denoted $\#\bsy$, is strictly greater than $K$. The law of the joint Markov chain $(\bsX_t,\bsY_t)_t$ under the parameter $\bstheta \in \bsTheta$ is denoted $\bar\bsP_{\bstheta}$ when initialised by the stationary distribution and $\bsP_{\bstheta}(\cdot \given \bsx_0)$ when assumed to start at the state $\bsx_0 \in \bbX^K$. The corresponding densities are written accordingly with lower-case letters.

The objective is to study the ratio $\bsp_{\bstheta}(\bsy_{1:n} \given \bsx_0) / \bsp_{\bstheta^*}(\bsy_{1:n} \given \bsx'_0)$ for any $\bsx_0 \in \bbX^K$ and any $\bsx'_0 \in \bbX^{K^*}$. 
The assumptions that are considered for this purpose are detailed in the next section.

\subsection{Assumptions and transferability}

In order to bring a better understanding of multi-object systems as a combination of single-object systems corrupted by clutter, assumptions are primarily made on individuals systems. The properties of multi-object systems will be deduced from these whenever this is possible.
\begin{enumerate}[label=\bfseries A.\arabic*,wide,series=hyp]
\item \label{it:boundTransition} The constants
\eqnl{eq:boundTransition}{
\tau_- = \inf_{\theta \in \Theta} \inf_{(x,x') \in \bbX^2} f_{\theta}(x \given x') \AND \tau_+ = \sup_{\theta \in \Theta} \sup_{(x,x') \in \bbX^2} f_{\theta}(x \given x')
}
satisfy $\tau_- > 0$ and $\tau_+ < \infty$.
\end{enumerate}
The condition on $\tau_-$ in Assumption~\ref{it:boundTransition} ensures that any point of the state space can be reached from any other point in a single time step (otherwise $f_{\theta}(x \given x') = 0$ would hold for at least one pair $(x,x') \in \bbX^2$) while the condition on $\tau_+$ ensures the transition is sufficiently regular when compared to the reference measure $\mu$, i.e.\ the transition should be diffuse (in the sense that there should be no concentration of probability mass on a single point of the state space). Under Assumption~\ref{it:boundTransition} it also holds that
\eqnl{eq:boundTransitionMulti}{
\tau_-^K \leq \bsf_{\bstheta}(\bsx \given \bsx') \leq \tau_+^K
}
for any $\bsx,\bsx' \in \bbX^K$, so that $\bsf_{\bstheta}$ straightforwardly satisfies the same type of conditions as $f_{\theta}$, since $S^{\T}$ is compact and hence $K$ is finite.

Let $\Pi_{\theta}$ be the transition kernel of the joint Markov chain $(X_t,Y_t)_t$ on $\bbX \times \bbY$ defined as $\Pi_{\theta}(x,y \given x',y') = g_{\theta}(y \given x) f_{\theta}(x \given x')$. The property \eqref{eq:boundTransitionMulti} is sufficient to ensure that the joint kernel defined as
\eqns{
\bsPi_{\bstheta}(\bsx,\bsy \given \bsx', \bsy') = \bsg_{\bstheta}(\bsy \given \bsx) \bsf_{\bstheta}(\bsx \given \bsx'),
}
for any $\bsx,\bsx' \in \bbX^K$ and any $\bsy,\bsy' \in \bbY^{\times}$ is positive Harris-recurrent and aperiodic. 

In the next assumption, the expectations $\bar\bbE_{\theta^*}[\cdot]$, $\bbE_{\psi^*}[\cdot]$ and $\bar\bbE_{\bstheta^*}[\cdot]$ are taken with respect to $\bar{P}_{\theta^*}$, $P_{\psi^*}$ and $\bar\bsP_{\bstheta^*}$ respectively, also $\Bi_p^k$ denotes the binomial distribution with success probability $p$ and $k$ trials.

\begin{enumerate}[hyp, resume]
\item \label{it:boundIntLikelihood} The constant
\eqnl{eq:likelihoodFiniteSup}{
\hat{b}^{\T}_+ \defeq \sup_{(\theta,x,y) \in \Theta\times\bbX\times\bbY} g_{\theta}(y \given x)
}
satisfies $\hat{b}^{\T}_+ < \infty$, the target- and clutter-related functions
\eqnsa{
b^{\T}_- : y \mapsto \inf_{\theta \in \Theta} \int g_{\theta}(y \given x) \d x & \AND b^{\T}_+ : y \mapsto \sup_{\theta \in \Theta} \int g_{\theta}(y \given x) \d x, \\
b_-^{\C} : y \mapsto \inf_{\psi} p_{\psi}(y) & \AND b_+^{\C} : y \mapsto \sup_{\psi} p_{\psi}(y),
}
satisfy
\eqnmla{eq:likelihoodConditions_positiveFinite}{
\label{eq:likelihoodConditions_positiveFinite1}
b^{\T}_-(y) &> 0 \AND b^{\T}_+(y) < \infty \\
b^{\C}_-(y) &> 0 \AND b^{\C}_+(y) < \infty
}
for any $y \in \bbY$ as well as
\eqnl{eq:likelihoodConditions}{
\bar\bbE_{\theta^*}[|\log b^{\T}_-(Y)|] < \infty \AND \bbE_{\psi^*}[|\log b^{\C}_-(Y)|] < \infty,
}
and it holds that
\eqnl{eq:logInfConvBiPo}{
\bar\bbE_{\bstheta^*}\big[ \big| \log \inf_{\bstheta \in \bsTheta} \Bi_{p_{\D}}^K * \Po_{\lambda}(\#\bsY) \big| \big] < \infty.
}
\end{enumerate}
Assumption~\ref{it:boundIntLikelihood} ensures that all points of the observation space $\bbY$ can be reached from at least some states in $\bbX$ via \eqref{eq:likelihoodConditions_positiveFinite1} although $g_{\theta}(y \given x) = 0$ might hold for some $(x,y) \in \bbX\times\bbY$. Equation \eqref{eq:likelihoodConditions} will ensure boundedness in the calculations related to identifiability. The upper bound \eqref{eq:likelihoodFiniteSup} of the likelihood function is also assumed to be finite so that no concentration of probability mass is allowed at any point of $\bbX\times\bbY$. It is demonstrated in the following lemma that the upper and lower bounds considered in Assumption~\ref{it:boundIntLikelihood} for a single target and for the clutter common distribution are sufficient to guarantee the same type of result for multiple targets. The proof is in \cref{proof:lem:boundedness}.

\begin{lemma}[Transfer of boundedness]
\label{lem:boundedness}
Under Assumption~\ref{it:boundIntLikelihood}, it holds that the constant
\eqns{
\hat\bsb_+ \defeq \sup_{\bstheta \in \bsTheta} \bigg( \sup_{(\bsx,\bsy) \in \bbX^K \times \bbY^{\times}} \bsg_{\bstheta}(\bsy \given \bsx) \bigg)
}
is finite and that the functions
\eqns{
\bsb_- : \bsy \mapsto \inf_{\bstheta \in \bsTheta} \int \bsg_{\bstheta}(\bsy|\bsx) \d \bsx \AND \bsb_+ : \bsy \mapsto \sup_{\bstheta \in \bsTheta} \int \bsg_{\bstheta}(\bsy|\bsx) \d \bsx,
}
verify $\bsb_-(\bsy) > 0$ and $\bsb_+(\bsy) < \infty$ for any $\bsy \in \bbY^{\times}$ as well as $\bar\bbE_{\bstheta^*}[| \log \bsb_-(\bsY)|] < \infty$.
\end{lemma}

An important result that follows from the assumptions introduced so far is the uniform forgetting of the conditional Markov chain: it can be proved under Assumptions~\ref{it:boundTransition} and \ref{it:boundIntLikelihood} that for any $k,l \in \bbN_0$ such that $k \leq l$ and any parameter $\bstheta \in \bsTheta$, it holds that
\eqns{
\int \bigg| \int \bar\bsp_{\bstheta}(\bsx_t \given \bsx_k, \bsy_{k+1:l}) \bsp(\bsx_k) \d \bsx_k - \int \bar\bsp_{\bstheta}(\bsx_t \given \bsx_k, \bsy_{k+1:l}) \bsp'(\bsx_k) \d \bsx_k \bigg| \d \bsx_t \leq \bsrho^{t-k}_{\bstheta},
}
for all $t \geq k$, all probability densities $\bsp$, $\bsp'$ on $\bbX^K$ and all sequences of observations $\bsy_{k+1:l}$, where $\bsrho_{\bstheta} \defeq 1 - (\tau_- / \tau_+)^K$. The $K$-target forgetting rate $\bsrho_{\bstheta}$ will generally be smaller than the single-target rate $1 - \tau_-/\tau_+$, although mixing is still guaranteed since $K$ is finite and hence $\bsrho_{\bstheta} \in [0,1)$. It is also possible to conclude about the pointwise convergence of the log-likelihood function to the function $\ell$ defined on $\bsTheta$ as follows
\eqns{
\ell : \bstheta \mapsto \bar\bbE_{\bstheta^*}\big[ \ell_{\bsY_{-\infty:0}}(\bstheta) \big],
}
where $\ell_{\bsy_{-\infty:0}}$ is defined on $\bsTheta$ for any realisation $\bsy_{-\infty:0}$ of the observation process as
\eqns{
\ell_{\bsy_{-\infty:0}} : \bstheta \mapsto \lim_{m \to \infty} \log \bar\bsp_{\bstheta}(\bsy_0 \given \bsy_{-m:-1}, \bsx_{-m-1}),
}
and this limit does not depend on $\bsx_{-m-1}$. Indeed, under Assumptions~\ref{it:boundTransition} and \ref{it:boundIntLikelihood}, it holds for all $K \in S^{\T}$ and all $\bsx_0 \in \bbX^K$ that
\eqnl{eq:pointwiseConvergence}{
\lim_{n \to \infty}  \dfrac{1}{n} \log \bsp_{\bstheta}(\bsY_{1:n} \given \bsx_0) = \ell(\bstheta), \qquad \bar\bsP_{\bstheta^*}\text{-a.s}.
}
This result shows that for any realisation $\bsy_{1:\infty}$ of the observation process, the empirical average $n^{-1} \log \bsp_{\bstheta}(\bsy_{1:n} \given \bsx_0)$ will converge to $\ell(\bstheta)$ irrespectively of the assumed initial state $\bsx_0$. A continuity assumption is required in order to turn the pointwise convergence result of \eqref{eq:pointwiseConvergence} into a uniform convergence result.
\begin{enumerate}[hyp, resume]
\item \label{it:continuity} For all $x,x' \in \bbX$ and all $y \in \bbY$, the mappings $\theta \mapsto f_{\theta}(x \given x')$, $\theta \mapsto g_{\theta}(y \given x)$ and $\psi \mapsto p_{\psi}(y)$ are continuous.
\end{enumerate}
It follows directly from Assumption~\ref{it:continuity} that for all $K \in S^{\T}$, all $\bsx, \bsx' \in \bbX^K$ and all $\bsy \in \bbY^{\times}$, the mappings $\bstheta \mapsto \bsf_{\bstheta}(\bsx \given \bsx')$ and $\bstheta \mapsto \bsg_{\bstheta}(\bsy \given \bsx)$ are continuous on the hyperplane of $\bsTheta$ made of parameters with a number of targets equal to $K$, since these mappings are sums and products of continuous functions. Although the continuity for the multi-target Markov kernel and likelihood function is limited to hyperplanes, the result of \cite[Lemma~4]{Douc2004} can be extended to: for all $\bstheta \in \bsTheta$
\eqns{
\lim_{\delta \to 0} \bar\bbE_{\bstheta^*}\bigg[ \sup_{|\bstheta' - \bstheta| \leq \delta} \big| \ell_{\bsY_{-\infty:0}}(\bstheta') - \ell_{\bsY_{-\infty:0}}(\bstheta) \big| \bigg] = 0,
}
where $|\cdot|$ is the $1$-norm on $\bsTheta$, since $\bstheta'$ and $\bstheta$ will be in the same hyperplane for $\delta$ small enough. The addition of the continuity assumption enables the derivation of the following result regarding the uniform convergence of the log-likelihood function:
Under Assumptions~\ref{it:boundTransition} to \ref{it:continuity}, it holds that
\eqnl{eq:uniformConvergence}{
\lim_{n \to \infty} \sup_{\bstheta \in \bsTheta} \sup_{\bsx_0 \in \bbX^K} \bigg| \dfrac{1}{n} \log \bsp_{\bstheta}(\bsY_{1:n} \given \bsx_0) - \ell(\bstheta) \bigg| = 0, \qquad \bar\bsP_{\bstheta^*}\text{-a.s}.
}

Since the conditional log-likelihood function $\log \bsp_{\bstheta}(\bsy_{1:n} \given \bsx_0)$ is continuous and uniformly bounded, it follows from \eqref{eq:uniformConvergence} that $\ell$ is also continuous on the hyperplanes of $\bsTheta$ of constant target number.

The following identifiability assumption is considered in order to show the consistency of the maximum likelihood estimator
\begin{enumerate}[hyp, resume]
\item \label{it:identifiability} $\bar{P}_{\theta} = \bar{P}_{\theta^*}$ if and only if $\theta = \theta^*$ and $P_{\psi} = P_{\psi^*}$ if and only if $\psi = \psi^*$
\end{enumerate}

Assumption~\ref{it:identifiability} is fundamental since there would be no chance to discriminate the true value $\theta^*$ among all the other possible $\theta \in \Theta\setminus\{\theta^*\}$ if some of these parameters did yield the same law for the observations. For instance, if the colour of the target is considered as a parameter but if the likelihood of the observations does not depend on this characteristics of the target, e.g.\ if the observations come from a radar, then any $\theta$ obtained by changing the colour in $\theta^*$ would induce a law $\bar{P}_{\theta}$ that is equal to $\bar{P}_{\theta^*}$ and Assumption~\ref{it:identifiability} would not be verified. It is shown in the next theorem that identifiability of the multi-target problem can be deduced from the identifiability of the single-target one under important special cases. The proof is in \cref{proof:thm:transferIdentifiability}.

\begin{theorem}[Transfer of identifiability]
\label{thm:transferIdentifiability}
Under Assumption~\ref{it:identifiability} it holds that
\begin{enumerate}[label=\alph*)]
\item if the true parameter $\bstheta^*$ is in $\bsTheta_{\lambda = 0}$, then it holds that $\bar\bsP_{\bstheta} = \bar\bsP_{\bstheta^*}$ if and only if $\bstheta = \bstheta^*$ for any $\bstheta \in \bsTheta_{\lambda = 0}$,
\item if the true parameter $\bstheta^*$ is in the subset $\bsTheta|_{K=1}$ of $\bsTheta$ made of parameters of the form $(\theta,K,p_{\D},\lambda,\psi)$ with $K = 1$, then it holds that $\bar\bsP_{\bstheta} = \bar\bsP_{\bstheta^*}$ if and only if $\bstheta = \bstheta^*$ for any $\bstheta \in \bsTheta|_{K=1}$.
\end{enumerate}
\end{theorem}

It is more challenging to prove that identifiability transfers to the whole parameter set $\bsTheta$ and this property is assumed to hold rather than demonstrated.
\begin{enumerate}[hyp, resume]
\item \label{it:MultiTargetIdentifiability} $\bar\bsP_{\bstheta} = \bar\bsP_{\bstheta^*}$ if and only if $\bstheta = \bstheta^*$.
\end{enumerate}

Assumption~\ref{it:MultiTargetIdentifiability} is not a stringent condition since \cref{thm:transferIdentifiability} shows that the single-target identifiability is sufficient to ensure multi-target identifiability in some important special cases. Moreover, if there exists a $\bstheta \neq \bstheta^*$ such that $\bar\bsP_{\bstheta} = \bar\bsP_{\bstheta^*}$ then $\bstheta$ has to satisfy very specific equations including
\eqns{
\Bi_{p_{\D}}^K * \Po_{\lambda} = \Bi_{p_{\D}^*}^{K^*} * \Po_{\lambda^*}.
}
Assumption~\ref{it:MultiTargetIdentifiability} would not hold for $p^*_{\D} = 0$ since identifiability w.r.t.\ $\theta^*$ and $K^*$ would clearly be lost in this case because of the absence of observations from the targets. The same remark can be made about $K^* = 0$ for the identifiability w.r.t.\ $\theta^*$ since there is obviously no way to learn about the dynamics and observation of the targets if none of them is present.

The different assumptions considered here are combined in the next section in order to prove the consistency of the maximum likelihood estimator.

\subsection{Consistency and asymptotic normality}

As a consequence of \eqref{eq:pointwiseConvergence} and by the dominated convergence theorem it holds that for any $\bstheta \in \bsTheta$, any infinite observation sequence $\bsy_{1:\infty}$ and any initial states $\bsx_0 \in \bbX^K$ and $\bsx'_0 \in \bbX^{K^*}$
\eqnmla{eq:limitRatio}{
\lim_{n\to\infty} \dfrac{1}{n} \log \dfrac{\bsp_{\bstheta}(\bsy_{1:n} \given \bsx_0)}{\bsp_{\bstheta^*}(\bsy_{1:n} \given \bsx'_0)} & = \ell(\bstheta) - \ell(\bstheta^*) \\
& = \lim_{m \to \infty} \bar\bbE_{\bstheta^*}\bigg[ \bar\bbE_{\bstheta^*}\bigg[ \log\dfrac{\bar\bsp_{\bstheta}(\bsY_0 \given \bsY_{-m:-1})}{\bar\bsp_{\bstheta^*}(\bsY_0 \given \bsY_{-m:-1})} \bigg| \bsY_{-m:-1} \bigg] \bigg] \leq 0,
}
where the inequality holds since the conditional expectations are Kullback-Leibler divergences. Yet, it could happen that some $\bstheta \in \bsTheta$ would verify $\bar\bsp_{\bstheta}(\bsY_0 \given \bsy_{-m:-1}) = \bar\bsp_{\bstheta^*}(\bsY_0 \given \bsy_{-m:-1})$ $\bar\bsP_{\bstheta^*}$-almost surely for all $m \in \bbN_0$ and for all $\bsy_{-m:-1}$, which would compromise identifiability. However, Assumption~\ref{it:MultiTargetIdentifiability} is equivalent to $\bstheta = \bstheta^*$ if and only if
\eqns{
\bar\bbE_{\bstheta^*}\bigg[ \log \dfrac{\bar\bsp_{\bstheta}(\bsY_{1:n})}{\bar\bsp_{\bstheta^*}(\bsY_{1:n})} \bigg] = 0, \qquad \forall n \geq 1.
}
The objective is to show that this, in turn, is equivalent to $\bstheta = \bstheta^*$ if and only if $\ell(\bstheta) - \ell(\bstheta^*) = 0$ since this is the term that appears in \eqref{eq:limitRatio}. Following the same line of arguments as \cite[Proposition~3]{Douc2004}, we find that under Assumptions~\ref{it:boundTransition} to \ref{it:continuity} and \ref{it:MultiTargetIdentifiability}, it holds that  $\ell(\bstheta) = \ell(\bstheta^*)$ if and only if $\bstheta = \bstheta^*$, from which we conclude that the considered approach allows for studying the identifiability of $\bstheta^*$. Applying the strict Jensen inequality to the conditional expectation in the r.h.s.\ of \eqref{eq:limitRatio}, it indeed follows that
\eqns{
\lim_{n\to\infty} \dfrac{1}{n} \log\dfrac{\bsp_{\bstheta}(\bsy_{1:n} \given \bsx_0)}{\bsp_{\bstheta^*}(\bsy_{1:n} \given \bsx'_0)} < 0,
}
for any $\bstheta \neq \bstheta^*$, which implies that the likelihood of the observation sequence $\bsy_{1:n}$ under the parameter $\bstheta$ decreases exponentially fast when compared to the likelihood under $\bstheta^*$, irrespectively of the assumed initial states $\bsx_0$ and $\bsx'_0$. Denoting $\hat\bstheta_{n,\bsx_0}$ the argument of the maximum of $\log \bsp_{\bstheta}(\bsy_{1:n} \given \bsx_0)$, the consistency of the maximum likelihood estimator can be expressed as in \cref{thm:consistencyMLE} below. This theorem also states the asymptotic normality of the estimator which makes use of the Fisher information. The latter involves differentiation with respect to the parameter $\bstheta$, however since the number of target $K$ is a natural number, differentiations has to be performed for a fixed $K$. This is what is understood by default when writing $\nabla_{\bstheta}$. Under assumptions, the Fisher information matrix can be expressed as
\eqns{
\bsI(\bstheta^*) = \lim_{n \to \infty} \dfrac{1}{n} \bar\bbE_{\bstheta^*} \Big[ \nabla_{\bstheta} \log \bar\bsp_{\bstheta^*}(\bsY_{1:n}) \cdot \nabla_{\bstheta} \log \bar\bsp_{\bstheta^*}(\bsY_{1:n})^{\tr} \Big],
}
where $\cdot^{\tr}$ is the matrix transposition.

\begin{theorem}
\label{thm:consistencyMLE}
Under Assumptions~\ref{it:boundTransition} to \ref{it:continuity} and \ref{it:MultiTargetIdentifiability}, it holds that
\eqns{
\lim_{n \to \infty} \hat\bstheta_{n,\bsx_0} = \bstheta^*
}
for any $\bsx_0 \in \bbX^K$ with $K \in \bbN$. Considering additionally Assumptions~\ref{it:asymptoticNormality1} to \ref{it:asymptoticNormality3} (see \cref{sec:assumptionAsymptoticNormality}) and assuming that $\bsI(\bstheta^*)$ is positive definite, it holds that
\eqns{
\sqrt{n} (\hat\bstheta_{n,\bsx_0} - \bstheta^*) \to \calN\big(0, \bsI(\bstheta^*)^{-1}\big),
}
for any $\bsx_0 \in \bbX^K$ and any $K \in \bbN$, where $\to$ denotes the convergence in distribution as $n$ tends to infinity and where $\calN(0,V)$ is the normal distribution with mean $0$ and variance $V$.
\end{theorem}

The proof of \cref{thm:consistencyMLE} follows from \cref{lem:boundedness} combined with \cite[Theorems~1 and 4]{Douc2004}. It can be demonstrated that the result of \cref{thm:consistencyMLE} also holds for the special parameter sets $\bsTheta_{\lambda=0}$, $\bsTheta_{p_{\D}=1}$ and $\bsTheta_{\lambda = 0, p_{\D}=1}$. These special-parameter sets will be used to understand the behaviour of the Fisher information matrix in simple cases in the next section.

\section{Analysis of the Fisher information}
\label{sec:FisherInformation}

\Cref{thm:consistencyMLE} guarantees the convergence of the maximum likelihood estimator under certain conditions and proves the asymptotic normality of the estimator, the variance of the latter being the inverse of the Fisher information matrix. It is therefore of interest to understand how the Fisher information behaves in different multi-target configurations.

This section is structured as follows: an equivalent observation model for which the Fisher information matrix is easier to study is introduced in \cref{ssec:atlObservationModel} and yields a characterisation of the configurations in which the information loss induced by data association uncertainty and detection failures is strictly positive. Qualitative estimates of the information loss are then obtained when isolating the different sources of loss from \cref{ssec:falseAlarm} to \cref{ssec:detectionFailure}. Each of these qualitative estimates are confirmed by numerical results on simulated data obtained by direct Monte Carlo integration of the original expression of the Fisher information, so as to confirm the validity of the derived alternative expressions.

Henceforth, if $A$ and $B$ are two square matrices of the same dimensions then $A \geq B$ is understood as $A - B \geq 0$, i.e.\ $A-B$ is positive semi-definite, and $A > B$ stand for $A - B > 0$, i.e.\ $A-B$ is positive definite.

\begin{example}
\label{ex:simpleFisher}
Assuming that $\bstheta^*$ is in $\bsTheta_{p_{\D}=1}$ and that the data association is known, the joint probability of the observations becomes
\eqnsml{
\bar\bsp_{\bstheta^*}(\bsy_{1:n}) = \prod_{t = 1}^n \bigg[ \Po_{\lambda^*}(M_t - K^*) \prod_{i = K^*+1}^{M_t} p_{\psi^*}\big(\bsy_{t,i}\big) \bigg] \\
\times \int \pi^{\times K^*}_{\theta^*}(\bsx_0) \prod_{t = 1}^n \prod_{i = 1}^{K^*} \Big[ g_{\theta^*}(\bsy_{t,i} \given \bsx_{t,i}) f_{\theta^*}(\bsx_{t,i} \given \bsx_{t-1,i}) \Big] \d\bsx_{0:n}.
}
The score is then found to be
\eqns{
\nabla_{\bstheta} \log \bar\bsp_{\bstheta^*}(\bsy_{1:n}) = \sum_{i=1}^{K^*} \nabla_{\bstheta} \log \bar{p}_{\theta^*}(\bsy_{1:n,i}) + \sum_{t=1}^n \bigg[ \dfrac{M_t - K^*}{\lambda^*} -1 + \sum_{i=K^*+1}^{M_t} \nabla_{\bstheta} \log p_{\psi^*}(\bsy_{t,i}) \bigg]
}
so that, because of the independence between the targets and clutter,
\eqns{
\bsI(\bstheta^*) = K^* I(\bstheta^*) + \dfrac{1}{\lambda^*} + \lambda^* I^{\C}(\bstheta^*),
}
with $I(\bstheta^*)$ and $I^{\C}(\bstheta^*)$ the Fisher information for the distribution of one target and one clutter point respectively, where the gradient is taken w.r.t.\ $\bstheta$, that is
\eqnsa{
I(\bstheta^*) & = \lim_{n \to \infty} \dfrac{1}{n} \bar\bbE_{\theta^*} \Big[ \nabla_{\bstheta} \log \bar{p}_{\theta^*}(Y_{1:n}) \cdot \nabla_{\bstheta} \log \bar{p}_{\theta^*}(Y_{1:n})^{\tr} \Big] \\
I^{\C}(\bstheta^*) & = \bbE_{\psi^*} \Big[ \nabla_{\bstheta} \log p_{\psi^*}(Y) \cdot \nabla_{\bstheta} \log p_{\psi^*}(Y)^{\tr} \Big].
}
\end{example}

In spite of its simplicity, \cref{ex:simpleFisher} yields important remarks: unsurprisingly, if there is no missing information and no data association uncertainty, the information increases with the number of targets. Similarly, if the Fisher information of the clutter distribution $p_{\psi^*}$ increases, then the overall information increases too. The interpretation for the Poisson parameter $\lambda^*$ is less straightforward, the main objective is however to study the Fisher information w.r.t.\ the targets rather than the false alarms so that it is of interest to compute the score without differentiating with respect to $\psi$ or $\lambda$.

Although the Fisher information becomes more difficult to compute when $p_{\D}^* \in (0,1)$, some conclusions can be drawn by focusing on the cardinality. Since the parameter $\theta$ does not affect the cardinality, only the term
\eqns{
\bbE[\nabla_{p_{\D}} \log \bsq_{\bstheta^*}(\bsD) \cdot \nabla_{p_{\D}} \log \bsq_{\bstheta^*}(\bsD)^{\tr}] = K^*/(p^*_{\D}(1-p^*_{\D}))
}
remains when computing the Fisher information matrix, with $\bsD$ the random variable induced by $\bsY$ on $\{0,1\}^{K^*}$. This term is minimal when $p^*_{\D} = 0.5$ and increases when $p^*_{\D}$ goes toward $0$ or $1$. This is not sufficient to conclude since the fact that information is lost when detection failures happen is not taken into account in the cardinality and the information is the same for, e.g.\ $p^*_{\D}$ equal to $0.99$ or $0.01$. Indeed, it is equally easy to estimate $p^*_{\D}$ when an observation is always or never received. For this reason, it is useful to consider the information w.r.t.\ $\theta^*$ only.

The objective will therefore be to characterise how the Fisher information
\eqnl{eq:FisherInfoTheta}{
\bsI(\theta^*) = \lim_{n \to \infty} \dfrac{1}{n} \bar\bbE_{\bstheta^*} \Big[ \nabla_{\theta} \log \bar\bsp_{\bstheta^*}(\bsY_{1:n}) \cdot \nabla_{\theta} \log \bar\bsp_{\bstheta^*}(\bsY_{1:n})^{\tr} \Big],
}
of a multi-object dynamical system behaves when compared to the information of the unperturbed system that excludes false alarms, detection failures and for which data association is known. We refer to the difference between \eqref{eq:FisherInfoTheta} and the latter as the \emph{information loss}. Since the Fisher information of the unperturbed system is a quantity that depends on the number of objects in the system, the aim is to express the information loss as a function of the single-object Fisher information matrix $I(\theta^*)$. The Fisher information matrix of the unperturbed multi-object system is clearly equal to $K^* I(\theta^*)$ because of the independence between the targets' observation in the absence of data association uncertainty. In order to compute $\bsI(\theta^*)$, we have to take the logarithm of the probability density function
\eqns{
\bar\bsp_{\bstheta^*}(\bsy_{1:n}, \bsx_{0:n}) = \pi^{\times K^*}_{\theta^*}(\bsx_0) \prod_{t=1}^n \big[ \bsg_{\theta^*}(\bsy_t \given \bsx_t ) \bsf_{\theta^*}(\bsx_t \given \bsx_{t-1}) \big].
}
However, the presence of a sum in the term $\bsg_{\bstheta^*}(\bsy_t \given \bsx_t )$ prevents from further analysing the Fisher information in a general setting. To avoid directly dealing with these sums, an equivalent observation model which depends explicitly on the assignment is introduced in the next section. This observation model is an important contribution since it allows us to understand the behaviour of the Fisher information for multi-target tracking.

\subsection{Alternative observation model}
\label{ssec:atlObservationModel}

Let $d_H$ be the Hamming metric on the symmetric group $\Sym(k)$ characterised by letting $d_H(\sigma,\sigma')$ be the number of points moved by $\sigma' \circ \sigma^{-1}$ for any given $\sigma,\sigma' \in \Sym(k)$. For instance, if $k=5$ and if $\sigma$ and $\sigma'$ are given in Cauchy's two-line notation as
\eqns{
\sigma =
\begin{pmatrix}
1 & 2 & 3 & 4 & 5 \\
1 & 5 & 2 & 4 & 3
\end{pmatrix}
\AND
\sigma' =
\begin{pmatrix}
1 & 2 & 3 & 4 & 5 \\
1 & 3 & 5 & 4 & 2
\end{pmatrix}
}
then $d_H(\sigma,\sigma') = 3$ since $\sigma' \circ \sigma^{-1}(i) \neq i$ if and only if $i$ is in the set $\{2,3,5\}$. Let $\oplus$ be the vector concatenation operator such that if $\bsy = [y_1,\dots,y_n]^{\tr} \in \bbY^n$ and $\bsy' = [y'_1,\dots,y'_m]^{\tr} \in \bbY^m$ then
\eqns{
\bsy\oplus\bsy' \defeq [y_1,\dots,y_n,y'_1,\dots,y'_m]^{\tr} \in \bbY^{n+m}.
}
Let $R_{\bsd}$ be the matrix of size $|\bsd| \times K^*$ such that $(R_{\bsd})_{i,j} = \delta_{j,r(i)}$ for any $\bsd \in \{0,1\}^{K^*}$, i.e.\ $R_{\bsd}$ has as many lines as there are detected targets and can be seen as a mask matrix that removes the observations of non-detected ones. Let $S_{\sigma}$ be the permutation matrix corresponding to $\sigma \in \Sym(k)$ for any $k \geq 1$, i.e.\ the matrix defined as
\eqns{
S_{\sigma} \defeq \begin{bmatrix} \bse_{\sigma(1)} \\ \vdots \\ \bse_{\sigma(K^*)} \end{bmatrix},
}
with $\bse_i$ the row vector with $1$ at the $i$\textsuperscript{th} position and $0$ elsewhere. The observation model with known data association is written as
\eqnl{eq:unperturbedModel}{
\bsY_t = \bsh(\bsX_t) + \bseta
}
with $\bsh$ and $\bseta$ the multi-target observation function and the observation noise respectively, where $\bseta$ is i.i.d.\ across its $K^*$ components. The false alarms are defined as a random variable $\hat\bsY$ in $\bbY^{\times}$, independent of $\bsY_t$, such that $\hat\bsY_i \sim p_{\psi}$ and $\hat\bsY_i$ is independent of $\hat\bsY_j$ for any $1 \leq i,j \leq \#\hat\bsY$. The observation model of interest can then be defined for given integers $\alpha > 0$ and $0 \leq \beta \leq K^*$ as
\eqnl{eq:perturbedModel}{
\bsY^{\alpha,\beta}_t = S_{\varsigma} \big( (R_{\bsD} \bsY_t) \oplus \hat\bsY \big),
}
where $\bsD$ is a random element of $B_{\beta} \defeq \{ \bsd' \in \{0,1\}^{K^*} \st |\one - \bsd'| \leq \beta \}$ having as a distribution the restriction $\bsq_{\bstheta}^{\beta}$ of $\bsq_{\bstheta}$ to $B_{\beta}$ and where $\varsigma$ is a random permutation drawn from the uniform law $u^{\alpha}_k$ with $k=\#\hat\bsY + |\bsd|$ on the set $A_k^{\alpha}$ defined by
\eqns{
A_k^{\alpha} \defeq \{ \sigma \in \Sym(k) \st d_H(\id,\sigma) \leq \alpha \}
}
with $\id$ denoting the identity function. Henceforth, the letter $\varsigma$ will be used for a random permutation and $\sigma$ for a realisation. The case $\alpha = 0$ is not considered to avoid redundancy: it holds that $A_k^0 = A_k^1 = \{\id\}$ for any $k \geq 1$ since permutations that are different from the identity move at least two points. The case of \cref{ex:simpleFisher} is recovered by considering $\alpha = 1$ and $\beta = 0$, i.e.\ $\varsigma = \id$ and $\bsD = \one$ almost surely, whereas the full data-association problem corresponds to the choice $\alpha = \infty$ and $\beta = \infty$. The cardinality of $A_k^{\alpha}$ is found to be
\eqns{
N_k^{\alpha} = \sum_{i=0}^{\alpha} \binom{k}{i} !i,
}
with $!i$ the subfactorial of $i$ which is equal to the number of \emph{derangements} of $i$ letters, where a derangement refers to a permutation that moves all the elements of its domain. The subfactorial $!i$ is defined via the same recurrence relation as the factorial $i!$, i.e.\ as $!i = (i-1)( !(i-1) + !(i-2) )$, but with the initialisation $!0 = 1$ and $!1 = 0$. The expression of $N_k^{\alpha}$ can be justified as follows: the number of permutations moving a number of points less or equal to $\alpha$ is also the number of permutations moving exactly $i$ points, i.e.\ the number of derangements of $i$ points, multiplied by the number of ways of picking $i$ points among $k$, for all $0 \leq i \leq \alpha$. It holds that $N^k_k = k!$ since
\eqns{
!n = n! - \sum_{i = 1}^n \binom{n}{i} !(n-i).
}
The alternative observation model \eqref{eq:perturbedModel} brings insight about the Fisher information matrix $\bsI^{\alpha,\beta}(\theta^*)$ corresponding to the observation model \eqref{eq:perturbedModel}, when compared to the unperturbed case. The corresponding information loss is defined as
\eqns{
\bsI^{\alpha,\beta}_{\loss}(\theta^*) \defeq K^* I(\theta^*) - \bsI^{\alpha,\beta}(\theta^*).
}
In some cases, the relative information loss $\bsI^{\alpha,\beta}_{\loss}(\theta^*) / (K^* I(\theta^*))$ will be used instead. The next theorem is the central result of this section, its proof can be found in \cref{proof:thm:informationLoss}.

\begin{theorem}
\label{thm:informationLoss}
Under Assumptions~\ref{it:boundTransition}, \ref{it:boundIntLikelihood} and \ref{it:asymptoticNormality1} to \ref{it:asymptoticNormality3}, the information loss $\bsI^{\alpha,\beta}_{\loss}(\theta^*)$ verifies $\bsI^{\alpha,\beta}_{\loss}(\theta^*) \geq 0$ for any $\alpha \geq 1$ and any $\beta \geq 0$, the inequality being strict if either $\alpha > 1$ or $\beta > 0$ and if $I(\theta^*) \neq 0$.
\end{theorem}

Notice that the condition $\alpha > 1$ would not be sufficient to make the inequality in \cref{thm:informationLoss} strict if $\lambda^*$ were equal to $0$ since data association might have no influence in some specific configurations, e.g.\ when the individual likelihood does not depend on the objects' state. \Cref{thm:informationLoss} does not provide a quantitative characterisation of the information loss. Doing so is challenging in the general case, yet, the behaviour of the information loss can be analysed for special cases, and such will be the objective in the remainder of this section.

One of the advantages with the modified observation model \eqref{eq:perturbedModel} is that the Fisher identity can be utilised as an alternative way of computing the score function based on the unobserved random variables in this model:
\eqnl{eq:FisherIdentity}{
\nabla_{\theta} \log \bar\bsp_{\bstheta}(\bsy_{1:n}) = \bar\bbE_{\bstheta} \big[ \nabla_{\theta} \log \bar\bsp_{\bstheta}(\bsY^{\alpha,\beta}_{1:n},\varsigma_{1:n},\bsD_{1:n},\bsX_{0:n}) \given \bsY^{\alpha,\beta}_{1:n} = \bsy_{1:n} \big],
}
where
\eqnsml{
\bar\bsp_{\bstheta}(\bsy_{1:n},\sigma_{1:n},\bsd_{1:n},\bsx_{0:n}) = \pi^{\times K}_{\theta}(\bsx_0) \prod_{t = 1}^n \bigg[ \Po_{\lambda}(M_t-|\bsd_t|) \\
\times \prod_{i = |\bsd_t|+1}^{M_t} p_{\psi}(\bsy_{t,\sigma_t(i)}) \prod_{i=1}^{|\bsd|} g_{\theta}(\bsy_{t,\sigma_t(i)} \given \bsx_{t,r(i)}) \prod_{i = 1}^K f_{\theta}(\bsx_{t,i} \given \bsx_{t-1,i}) u_{M_t}^{\alpha}(\sigma_t) \bsq_{\bstheta}^{\beta}(\bsd_t) \bigg].
}
The simplification of the expression of $\nabla_{\theta} \log \bar\bsp_{\bstheta}(\bsy_{1:n})$ is only notational. The random variables $\varsigma_{1:n}$, $\bsD_{1:n}$ and $\bsX_{0:n}$ are conditioned on the event $\bsY^{\alpha,\beta}_{1:n} = \bsy_{1:n}$ in \eqref{eq:FisherIdentity}, so that their respective distributions are now the conditional distributions given the observations, which are more complex than their priors. Yet, the Fisher identity enabled to move the sums and integrals outside of the logarithm, hence making easier the analysis of the Fisher information matrix.

\subsection{Single static target with false alarm}
\label{ssec:falseAlarm}

Consider the case of one almost-surely detected static target with state $x \in \bbX$ which observation is corrupted by false alarms and unknown data association. The corresponding $\bstheta^*$ is in the hyperplane $\bsTheta_{p_{\D}=1}|_{K=1}$ of the special parameter set $\bsTheta_{p_{\D} = 1}$ composed of parameters for which $K=1$. It is sufficient to study one time step since the observations at different times become independent in this case and it holds that
\eqnsa{
\bsI(\theta^*) & = \lim_{n \to \infty} \dfrac{1}{n} \bar\bbE_{\bstheta^*} \Big[ \nabla_{\theta} \log \bar\bsp_{\bstheta^*}(\bsY_{1:n}) \cdot \nabla_{\theta} \log \bar\bsp_{\bstheta^*}(\bsY_{1:n})^{\tr} \Big] \\
& = \bar\bbE_{\bstheta^*} \Big[ \nabla_{\theta} \log \bar\bsp_{\bstheta^*}(\bsY_1) \cdot \nabla_{\theta} \log \bar\bsp_{\bstheta^*}(\bsY_1)^{\tr} \Big].
}
Making use of the Fisher identity \eqref{eq:FisherIdentity}, the Fisher information matrix $\bsI^{\infty,0}(\theta^*)$ can be expressed as
\eqns{
\bsI^{\infty,0}(\theta^*) = \bbE_{\bstheta^*}\bigg[ \sum_{i,j = 1}^M c_i(\bsY)c_j(\bsY) \nabla_{\theta} \log g_{\theta^*}(\bsY_i \given x) \cdot \nabla_{\theta} \log g_{\theta^*}(\bsY_j \given x)^{\tr} \bigg],
}
where $M = \#\bsY$ and where
\eqns{
c_i(\bsy) = \sum_{\substack{\sigma\in \Sym(M) \\ \sigma(1) = i}} u_M(\sigma \given \bsy) = \dfrac{g_{\theta^*}(\bsy_i \given x) / p_{\psi^*}(\bsy_i)}{\sum_{j = 1}^M g_{\theta^*}(\bsy_j \given x) / p_{\psi^*}(\bsy_j)}.
}
Identifying the parameter $\theta^*$ is most challenging when the distribution of the false alarm is equal to the one of the target-originated observation at $\theta^*$, i.e.\ $p_{\psi^*} = g_{\theta^*}( \cdot \given x)$, since all the observations will look alike for $\theta$ close to $\theta^*$.

In this case it holds that $c_i(\bsY) = 1/M$ for any $1 \leq i \leq M$ so that
\eqnsa{
\bsI^{\infty,0}(\theta^*) & = \sum_{m \geq 1} \dfrac{\Po_{\lambda^*}(m-1)}{m^2} \sum_{i = 1}^m \bbE_{\bstheta^*}\big[ \nabla_{\theta} \log g_{\theta^*}(\bsY_i \given x) \cdot \nabla_{\theta} \log g_{\theta^*}(\bsY_i \given x)^{\tr} \given m \big] \\
& = \bbE\Big[\dfrac{1}{N+1}\Big] I(\theta^*)
}
where the expectation is taken w.r.t.\ the random variable $N \sim \Po_{\lambda^*}$. It follows that the relative information loss is equal to $\bbE[N/(N+1)]$ so that it is strictly increasing with $\lambda^*$ and tends to $1$ when $\lambda^*$ tends to infinity. This result is supported by the experiments displayed in \cref{fig:lossFalseAlarm} where the observation of one static target in $\bbX = \bbR$ at $x = 0$ is corrupted by false alarms. The observation model is assumed to be linear and Gaussian with variance $\theta$ such that $\theta^* = 1$. Cases where the false alarm is uniform over the subset $[-a,a]$ with $a \in \{5,10,25,50,100\}$ are also considered. The scenario where the false alarm is distributed in the same way as the target-originated observation at $\theta^*$, i.e.\ $p_{\psi^*} = g_{\theta^*}(\cdot \given x)$, is also confirmed to be the worst-case scenario.

\begin{figure}
\centering
\includegraphics[width=.5\textwidth,trim = 110pt 265pt 120pt 285pt, clip]{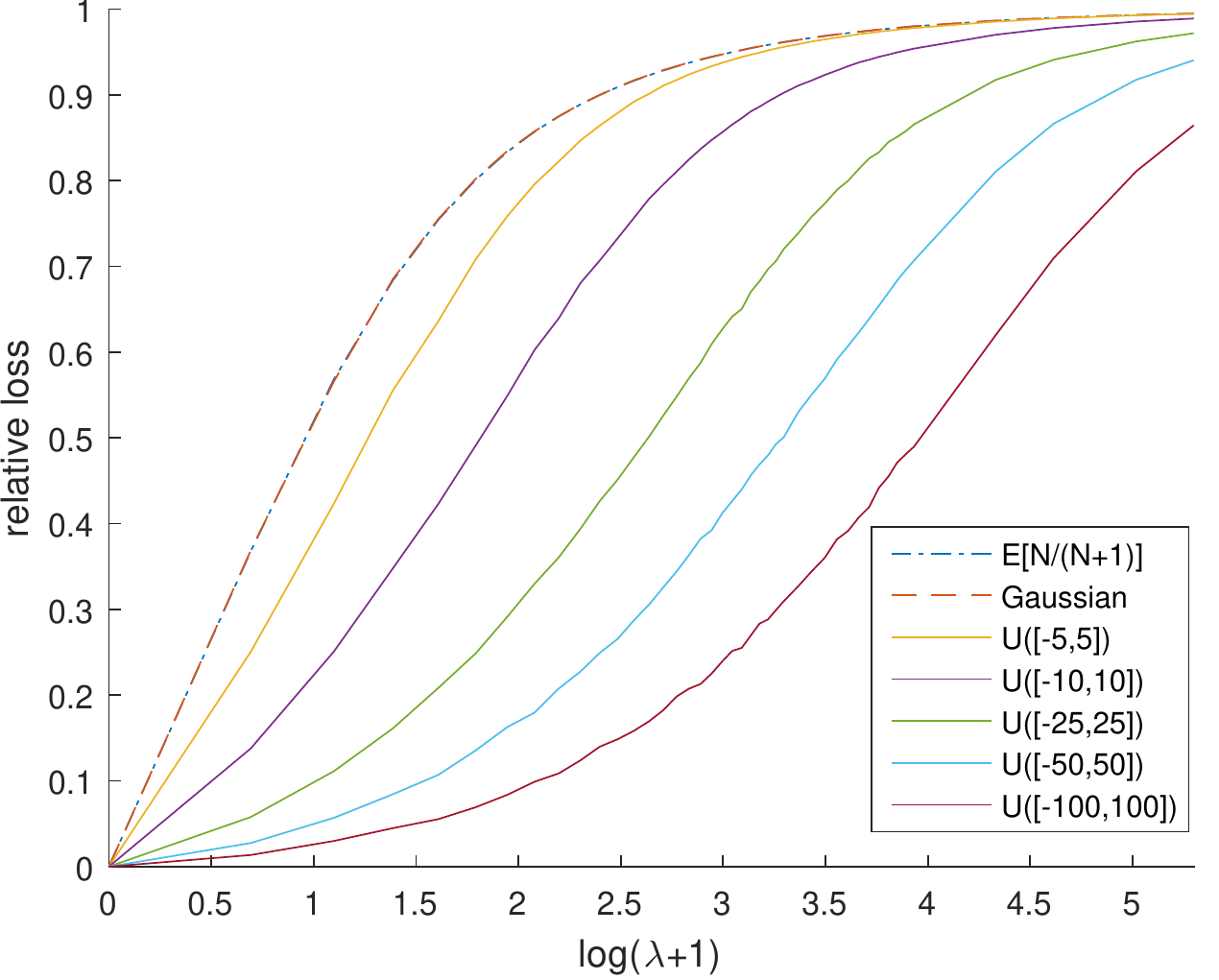}
\caption{Information loss as a function of the Poisson parameter $\lambda$ in log-scale, calculated with $5\times10^5$ samples (Gaussian: worst-case scenario; $U([-a,a])$: uniform distribution over $[-a,a]$).}
\label{fig:lossFalseAlarm}
\end{figure}

In the next two sections, the focus will be on understanding the role played specifically by unknown data association and detection failures.

\subsection{Unknown data association}
\label{ssec:informationDataAssociation}

In order to set the focus on data association, it is assumed that $\bstheta^*$ belongs to the special parameter set $\bsTheta_{\lambda=0,p_{\D}=1}$. In these conditions, the joint probability of the observations and states becomes
\eqnlml{eq:jointDistributionSumAssociations}{
\bar\bsp_{\bstheta^*}(\bsy_{1:n}, \bsx_{0:n}) = \pi^{\times K^*}_{\theta^*}(\bsx_0) \\
\times \prod_{t = 1}^n \sum_{\sigma \in \Sym(K^*)} \bigg[ \prod_{i = 1}^{K^*} \big[ g_{\theta^*}(\bsy_{t,\sigma(i)} \given \bsx_{t,i}) f_{\theta^*}(\bsx_{t,i} \given \bsx_{t-1,i}) \big] u_{K^*}(\sigma) \bigg] .
}
The sum in the previous expression makes it difficult to directly compute the Fisher information matrix. Some insight about it can however be obtained by considering static objects as in the following example.

\begin{example}
\label{ex:staticTargets}
Let $x_1,\dots,x_{K^*}$ be the known position of $K^*$ static objects. The joint distribution of the observations is then found to be
\eqns{
\bsp_{\bstheta^*}(\bsy_{1:n}) = \prod_{t = 1}^n \sum_{\sigma \in \Sym(K^*)} \bigg[ \prod_{i = 1}^{K^*} g_{\theta^*}(\bsy_{t,\sigma(i)} \given x_i) u_{K^*}(\sigma) \bigg].
}
In this simplified setting, we can assume that $g_{\theta^*}$ has finite support so that the objects' state can be chosen far enough from each other for $\prod_{i = 1}^{K^*} g_{\theta^*}(\bsY_{t,\sigma(i)} \given x_i) = 0$ to hold $\bbP$-a.s.\ whenever $\sigma \neq \id$. In this case, and as expected, there is no loss of information when compared to the case with known data association. A less intuitive result can be found when all the objects' state are equal to a given $x \in \bbX$. In this situation, it holds that all permutations are equally probable so that
\eqns{
\bsp_{\bstheta^*}(\bsy_{1:n}) = \prod_{t = 1}^n \prod_{i = 1}^{K^*} g_{\theta^*}(\bsy_{t,i} \given x),
}
and once again, there is no loss of information. These two cases correspond to extreme configurations where the uncertainty on the data association is either resolvable or irrelevant.
\end{example}

The Fisher identity can be used to provide an expression of the Fisher information for static objects as follows. For any fixed $x_1,\dots,x_{K^*}$, the Fisher information for $\alpha = \infty$ (fully unknown association) and $\beta = 0$ can be deduced from
\eqnsa{
\nabla_{\theta} \log \bsp_{\bstheta}(\bsy) & = \bbE_{\bstheta}\big[ \nabla_{\theta} \log \bsp_{\bstheta}(\bsY,\varsigma) \given \bsY = \bsy \big] \\
& = \sum_{\sigma \in \Sym(K)} \sum_{i=1}^K \nabla_{\theta} \log g_{\theta}(\bsy_{\sigma(i)} \given x_i) u_{K}(\sigma | \bsy).
}
The Fisher information matrix $\bsI^{\alpha,\beta}(\theta^*)$ with $\alpha = \infty$, $\beta = 0$ and without false alarm is found to be
\eqnsa{
\bsI^{\infty,0}(\theta^*) & = \bbE_{\bstheta^*}\big[ \nabla_{\theta} \log \bsp_{\bstheta^*}(\bsY^{\infty,0}) \cdot \nabla_{\theta} \log \bsp_{\bstheta^*}(\bsY^{\infty,0})^{\tr} \big] \\
& = \sum_{\sigma,\sigma' \in \Sym(K^*)} \sum_{i,j = 1}^{K^*} \bbE_{\theta^*}\big[ u_{K^*}(\sigma | \bsY) u_{K^*}(\sigma' | \bsY) \Sco_i(\bsY_{\sigma(i)}) \cdot \Sco_j(\bsY_{\sigma'(j)})^{\tr} \big] \\
& = \sum_{i,j,k,l = 1}^{K^*} \bbE_{\theta^*}\big[ c_{i,k}(\bsY)c_{j,l}(\bsY) \Sco_i(\bsY_k) \cdot \Sco_j(\bsY_l)^{\tr} \big]
}
with $\Sco_i(y) = \nabla_{\theta} \log g_{\theta^*}(y \given x_i)$ for any $y \in \bbY$ and with
\eqnsa{
c_{i,k}(\bsy) & = \sum_{\substack{\sigma \in \Sym(K^*) \\ \sigma(i) = k}} u_{K^*}(\sigma | \bsy) \\
& = g_{\theta^*}(\bsy_k \given x_i) \sum_{\substack{\sigma \in \Sym(K^*) \\ \sigma(i) = k}} \prod_{j\neq i} g_{\theta^*}(\bsy_{\sigma(j)} \given x_j) \bigg( \sum_{\sigma \in \Sym(K^*)} \prod_{j = 1}^{K^*} g_{\theta^*}(\bsy_{\sigma(j)} \given x_j) \bigg)^{-1}
}
for any $\bsy \in \bbY^{K^*}$ and any $i,k \in \{1,\dots,K^*\}$. The term $c_{i,k}(\bsy)$ is the conditional probability for the object with state $x_i$ to have generated observation $k$ given all observations $\bsy$. 

In order to obtain a quantitative characterisation of the information loss, a special likelihood has to be introduced. We consider an observation model of the same form as the one displayed in \cref{fig:normalUniform}, i.e.\ such that $\bbY$ is compact and there exists a collection of disjoint subsets $\{B_i\}_{i=1}^{K^*}$ of $\bbY$ such that $g_{\theta}( \cdot \given x_i)$ uniformly distributes a probability mass $\epsilon > 0$ outside of $B_i$. An example of such a distribution is given in \cref{fig:normalUniform} for two objects. Then, for $K$ objects,
\eqnl{eq:infSimpleLikelihood}{
\bsI^{\infty,0}(\theta^*) = \sum_{i,j,k,l=1}^K E_{i,j}^{k,l}(\theta^*)
}
with $E_{i,j}^{k,l}(\theta^*) \defeq \bbE_{\theta^*}\big[ c_{i,k}(\bsY)c_{j,l}(\bsY) \Sco_i(\bsY_k) \cdot \Sco_j(\bsY_l)^{\tr} \big]$ for any $i,j,k,l \in \{1,\dots,K\}$. The objective is now to understand the behaviour of $\bsI^{\infty,0}(\theta^*)$ when $K$ is large. The order of the term $c_{i,k}(\bsy)$ is in $O(1)$ when $i = k$ and in $O(K^{-1})$ when $i \neq k$. The order of the summand in \eqref{eq:infSimpleLikelihood} can then be determined for the different values of $i,j,k,l$. For instance:
\begin{itemize}[wide]
\item If $i \neq k \neq l \neq j$ then
\eqnl{eq:exOrderTerm}{
E_{i,j}^{k,l}(\theta^*) = 
\dfrac{\epsilon^2}{|\bbY\setminus B_k|^2} \int_{C_{i,j}^{k,l}} c_{i,k}(\bsy)c_{j,l}(\bsy) \dfrac{\nabla_{\theta}\, g_{\theta^*}(\bsy_k \given x_i) \cdot \nabla_{\theta}\, g_{\theta^*}(\bsy_l \given x_j)^{\tr}}{g_{\theta^*}(\bsy_k \given x_i)g_{\theta^*}(\bsy_l \given x_j)} \d \bsy,
}
where $C^{k,l}_{i,j} \defeq \{\bsy \in \bbY^K \st \bsy_k \in B_i \And \bsy_l \in B_j\}$, because $g_{\theta^*}(y \given x_k) = \epsilon/|\bbY\setminus B_k|$ for all $y \notin B_k$ and because $B_i \cap B_k = \emptyset$ since $i \neq k$. When $K$ increases, $\bbY$ needs to be augmented at least linearly to ensure that the family $\{B_i\}_{i=1}^K$ is disjoint and \eqref{eq:exOrderTerm} shows inverse proportionality with $|\bbY|^2$, so that it is of order $O(K^{-4})$ at most. There are $O(K^4)$ terms of this form in the sum in the r.h.s.\ of \eqref{eq:infSimpleLikelihood} so that the sum of these terms is of order $O(1)$ at most.
\item If $k=l$ and $i \neq j$ then $E_{i,j}^{k,l}(\theta^*) = 0$ since in this case it holds that $\Sco_i(\bsy_k) \cdot \Sco_j(\bsy_k)^{\tr} = 0$ for any $\bsy \in \bbY^K$ which follows from the facts that $\Sco_i(y) \neq 0$ when $y \in B_i$ only and that $B_i \cap B_j = \emptyset$.
\item If $i=j=k=l$ then
\eqns{
E_{i,j}^{k,l}(\theta^*) = \int \ind{B_i}(\bsy_i) c_{i,i}(\bsy)^2 \dfrac{\nabla_{\theta}\, g_{\theta^*}(\bsy_i \given x_i) \cdot \nabla_{\theta}\, g_{\theta^*}(\bsy_i \given x_i)^{\tr}}{g_{\theta^*}(\bsy_i \given x_i)} \d \bsy,
}
which does not depend on $K$ or $|\bbY|$ and is therefore of order $O(1)$.
\end{itemize}

\begin{figure}
\centering
\includegraphics[width=.5\textwidth,trim = 110pt 335pt 120pt 355pt, clip]{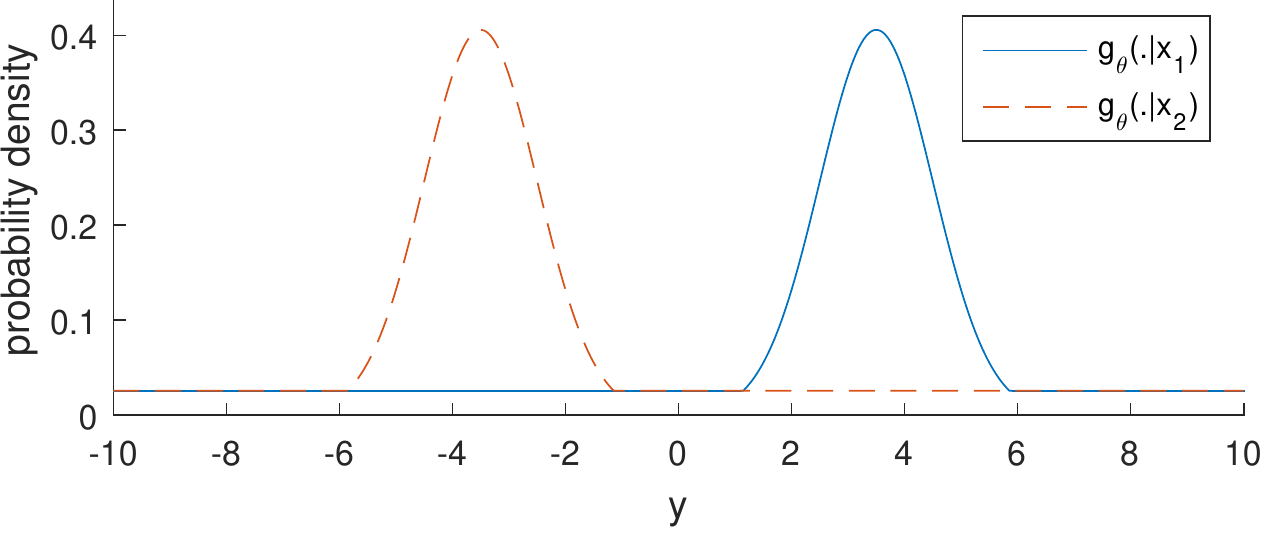}
\caption{Example of likelihood with two objects at states $x_1$ and $x_2$.}
\label{fig:normalUniform}
\end{figure}

Following the same principles for the other values of $i,j,k,l$, we find that $\bsI^{\infty,0}(\theta^*)$ is of order $O(K)$. Since the information in the idealised observation model, i.e.\ when data association is known, is equal to $K I(\theta^*)$, it follows that the relative loss is constant. In other words, for a large number of targets, adding more targets increases the information at the same rate as in the idealised model.

\paragraph{Validation via simulations} The special likelihood is taken of the form
\eqns{
g_{\theta}(y \given x_k) = 
\begin{cases*}
\calN(y; x_k+m,1) & if $y \in B_k \defeq (x_k+m-r, x_k+m+r)$ \\
\epsilon/|\bbY\setminus B_k| & otherwise,
\end{cases*}
}
with $\epsilon = 0.1$ and with $r$ characterised by $\int_{B_k} \calN(y; x_k,1) \d y = 1 - \epsilon$ via $B_k$. In this case, the displacement $m$ is considered as the parameter $\theta$ and the true value is $\theta^* = 0$. The relative information loss associated with this likelihood is displayed in \cref{fig:lossDataAssociation2} under two different configurations. The first one (\emph{Constant observation space} in the figure) corresponds to the case where the observation space is large enough to meet the requirements associated with \eqref{eq:infSimpleLikelihood}; the relative loss can be seen to increase linearly with the number of targets. The second case (\emph{Adaptive observation space} in the figure) corresponds to the case where the observation space has to be augmented to fit new targets and shows a constant relative information loss. This last result is consistent with the conclusion above that the information loss is of the same order as the number of targets when the observation space has to be augmented.

\begin{figure}
    \centering
    \begin{subfigure}[b]{0.49\textwidth}
      \includegraphics[width=\textwidth,trim = 105pt 265pt 120pt 285pt, clip]{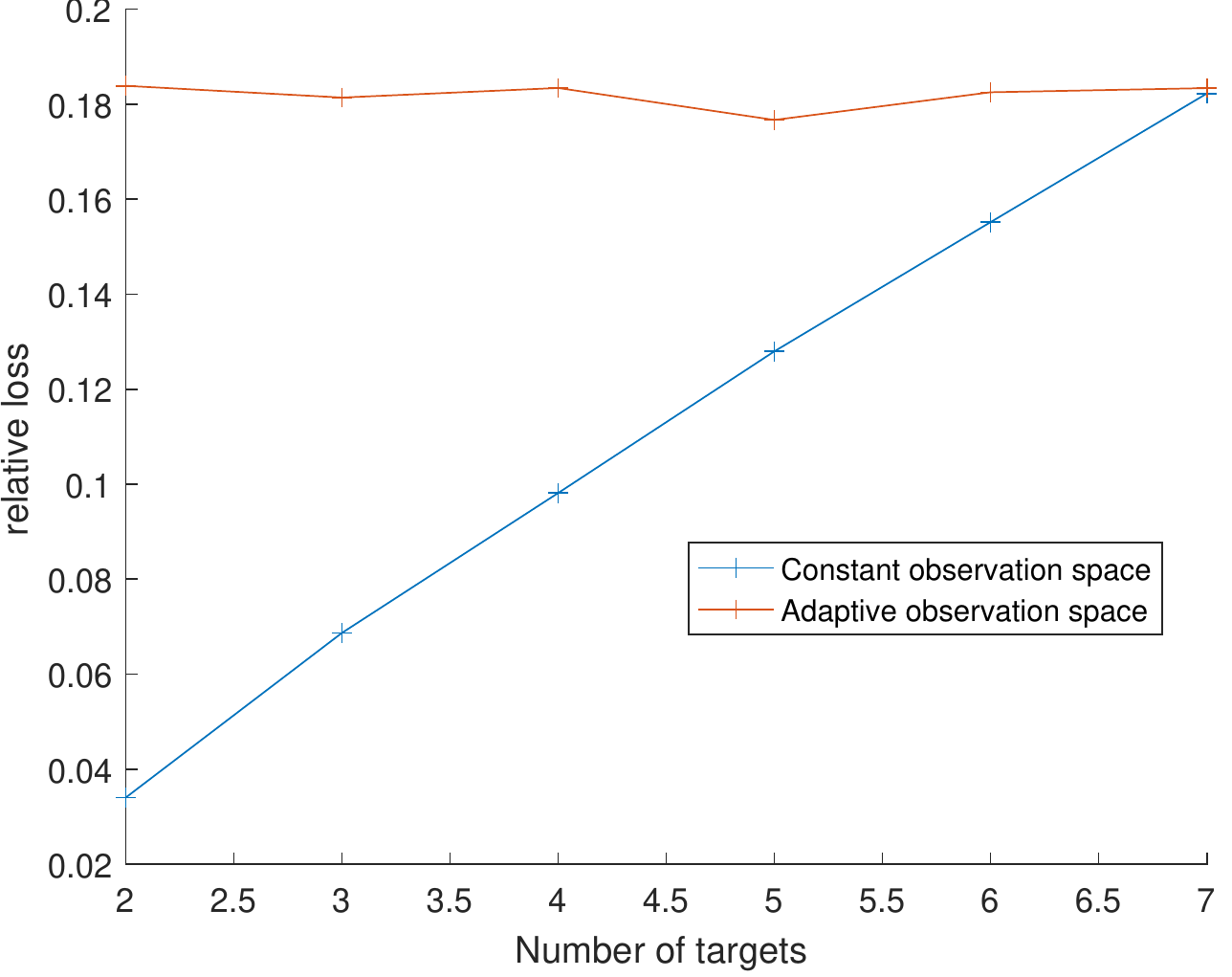}
      \caption{For a varying number of objects for the special likelihood ($10^5$ MC runs).}
      \label{fig:lossDataAssociation2}
    \end{subfigure} 
    \begin{subfigure}[b]{0.49\textwidth}
      \includegraphics[width=\textwidth,trim = 110pt 265pt 120pt 285pt, clip]{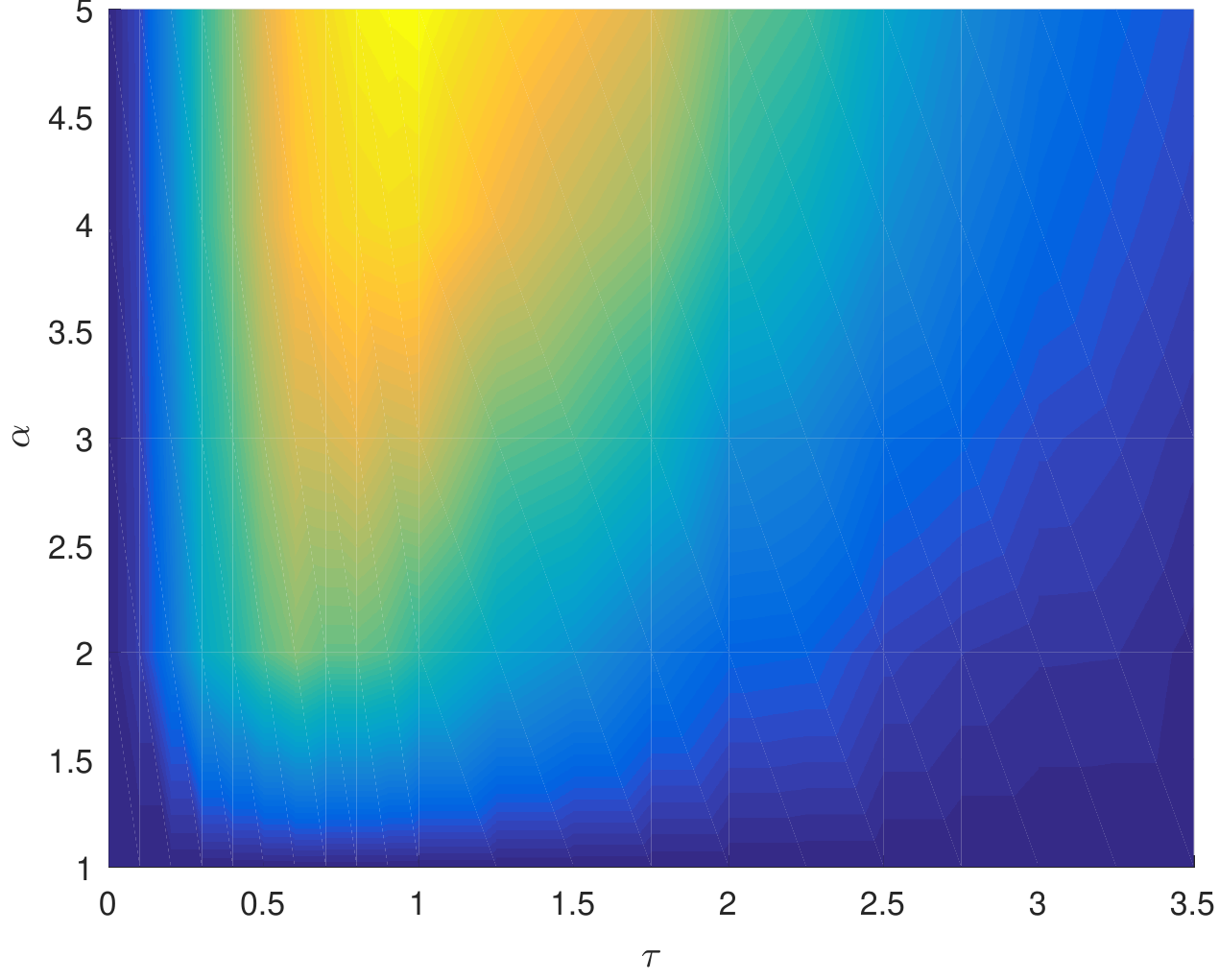}
      \caption{For varying association uncertainty $\alpha$ and spatial separation $\tau$ ($10^4$ MC runs).}
      \label{fig:lossDataAssociation}
    \end{subfigure}
    \begin{subfigure}[b]{0.49\textwidth}
      \includegraphics[width=\textwidth,trim = 105pt 265pt 120pt 285pt, clip]{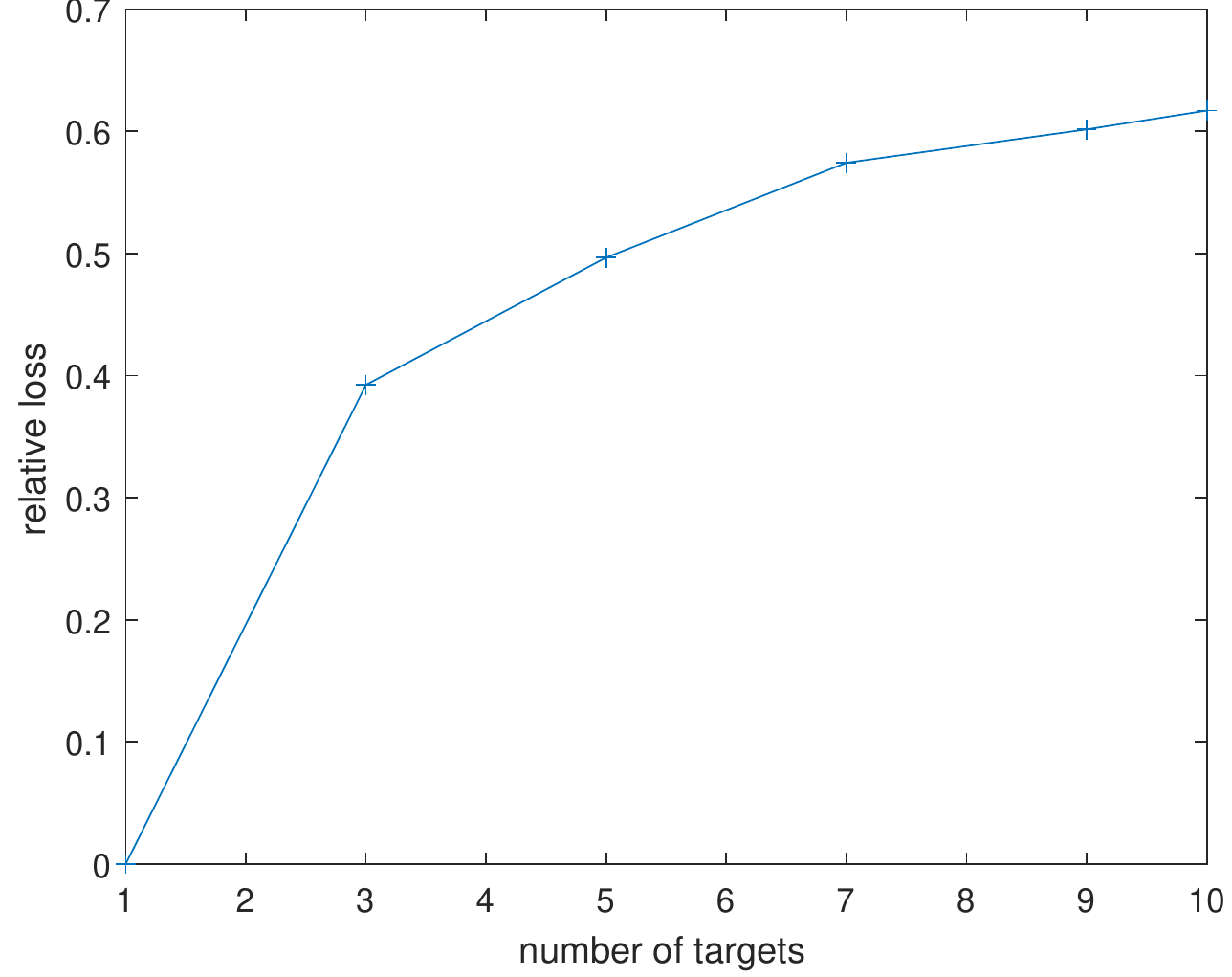}
      \caption{For a varying number of objects with separation $\tau = 1$ and $\alpha=\infty$ ($10^4$ MC runs).}
      \label{fig:lossDataAssociation3}
    \end{subfigure}
    \begin{subfigure}[b]{0.49\textwidth}
      \includegraphics[width=\textwidth,trim = 110pt 265pt 120pt 285pt, clip]{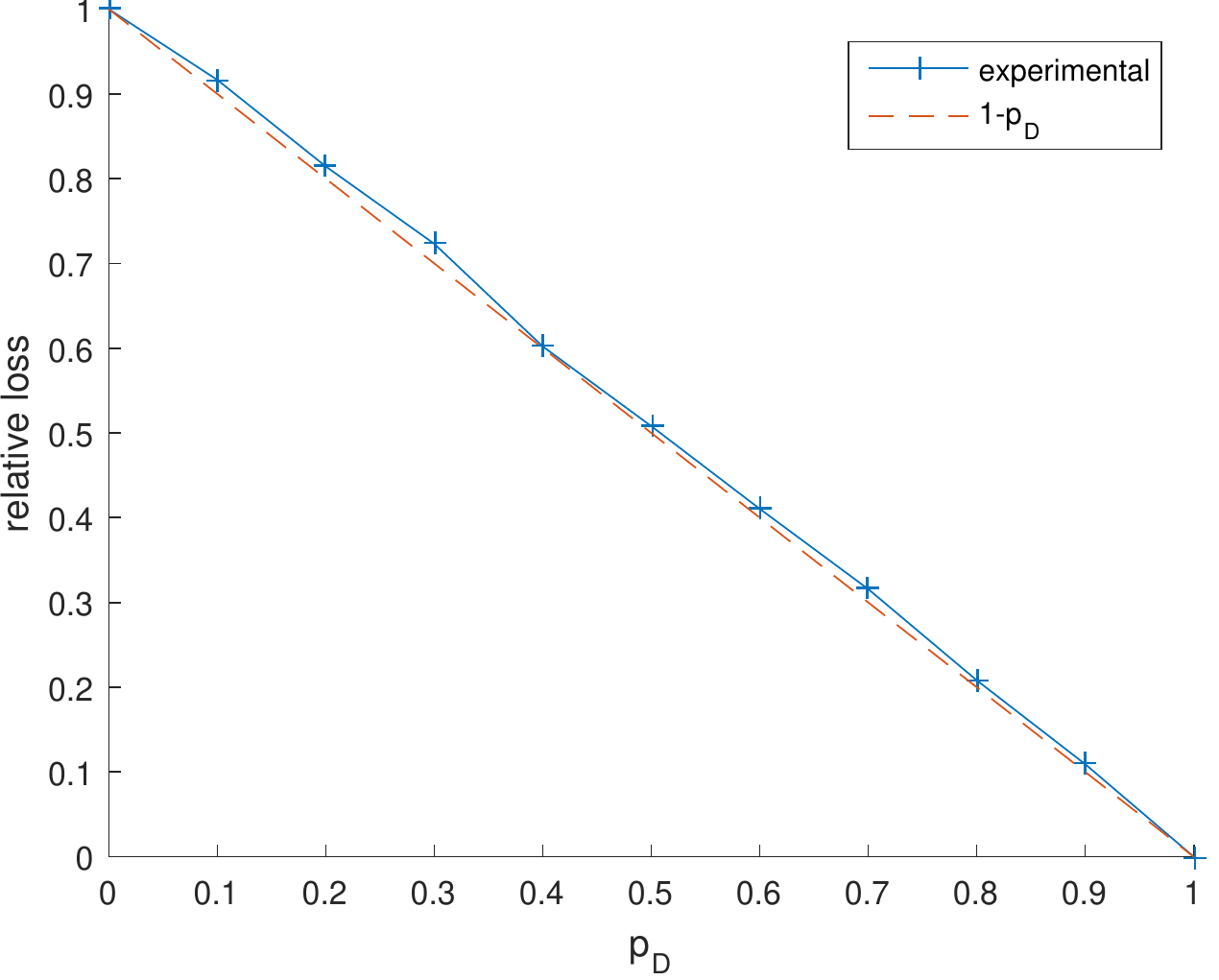}
      \caption{For a varying probability of detection $p_{\D}$, compared to $1-p_{\D}$.}
      \label{fig:lossDetectionFailure}
    \end{subfigure}
    \caption{Information loss with association uncertainty (\ref{fig:lossDataAssociation2}-\ref{fig:lossDataAssociation3}) or detection failures (\ref{fig:lossDetectionFailure}).}
\end{figure}

\paragraph{Further simulations} Five static objects on $\bbX = \bbR$ at positions $x_i = \tau (i-3)$ with $i \in \{1,\dots,5\}$ are observed via a linear Gaussian model with variance equal to $1$. The objective is to understand how the Fisher information matrix $\bsI^{\alpha,0}(\theta^*)$ evolves with $\alpha$ and with the position of the objects. It is assumed that $\theta$ parametrises the variance of the Gaussian observation model only, so that $\bsI^{\alpha,0}(\theta^*)$ is a scalar. The relative information loss is displayed in \cref{fig:lossDataAssociation} and confirms the intuition that the information loss increases with $\alpha$, except in the case $\alpha = 1$ where there is no loss by definition since $A_k^1 = \{\id\}$ for any $k \geq 1$ so that the data association is known in this case. Also, the loss is increased when the individual likelihoods overlap while being increasingly different and then decreases when the overlap becomes negligible. The maximum is reached when $\tau = 1$, that is when the distance $|x_i - x_{i-1}|$ between two consecutive objects is $1$ for any $i \in \{2,\dots,5\}$. The fact that there is no loss when $\tau = 0$ follows from the irrelevance of data association uncertainty when all objects are at the same position, as explained in \cref{ex:staticTargets}. To better understand the behaviour w.r.t.\ the number of targets, \cref{fig:lossDataAssociation3} displays the relative information loss for $1$ to $10$ targets in the case of full data association uncertainty  with $\tau = 1$.

The results for the two sets of simulations are consistent and show the same trend: the relative information loss increases with the number of targets but tend to stabilise. To sum up, there is no loss for $1$ target by construction, the loss is linear in the number of targets when there are sufficiently many, and it increases the fastest during the transition between these two modes.

\subsection{Detection failures}
\label{ssec:detectionFailure}

In this section, the case of detection failures is analysed when assuming that there are no false alarms, that is when $\bstheta^*$ is in the special-parameter set $\bsTheta_{\lambda=0}$.
To establish our main result in this section (\cref{thm:informationLossPd}), we will use the concept of missing information (see for instance \cite{Dean2014} in the context of Approximate Bayesian Computation).

\begin{theorem}
\label{thm:informationLossPd}
Assuming $\bstheta^* \in \bsTheta_{\lambda=0}$, the information loss $\bsI^{\alpha,\beta}_{\loss}(\theta^*)$ for known data association with unconstrained detection failures, i.e\ for $\alpha = 1$, $\beta = \infty$, is found to be
\eqns{
\bsI^{1,\infty}_{\loss}(\theta^*) \defeq (1-p^*_{\D}) K^* I(\theta^*).
}
\end{theorem}

The proof can be found in \cref{proof:thm:informationLossPd}. It follows from \cref{thm:informationLossPd} that in the considered configuration the Fisher information matrix $\bsI^{1,\infty}(\theta^*)$ can be made arbitrary close to $0$ by making $p^*_{\D}$ tend to~$0$. Also, there is no loss at all when $p^*_{\D} = 1$, as expected. In order to verify the result of \cref{thm:informationLossPd} in practice, a single-object scenario with detection failures and without false alarms is considered. The object starts at time $t=0$ from the position $x_0 = 0$ and evolves in $\bbX = \bbR$ according to a random walk with standard deviation $0.1$ until time $n=50$. The observation is linear and Gaussian with variance equal to $1$. The integral over the state space in the expression of the score is computed by Monte Carlo simulation with $10^3$ samples while the expectation in the Fisher information utilises $10^4$ samples. The relative information loss is displayed in \cref{fig:lossDetectionFailure} and confirms the coefficient $1-p_{\D}$ found analytically in \cref{thm:informationLossPd}. The next example shows how the Fisher information evolves in general when adding new objects without involving them in data association uncertainty.

\begin{example}
The Fisher information $\bsI^{\alpha,\beta}(\theta^*,K)$ of a $K$-object problem can be related to the information $\bsI^{\alpha,\beta}(\theta^*,K+N)$ where the $N$ new objects are not perturbed by data-association uncertainties, i.e.\ when the random variable $\varsigma$ in the observation model \eqref{eq:perturbedModel} verifies $\varsigma|_D = \id$ almost surely with $ D \defeq \{|\bsd_{1:K+1}|, \dots, |\bsd_{1:N}|\}$. It then follows from \cref{thm:informationLossPd} that
\eqns{
\bsI^{\alpha,\beta}(\theta^*,K+N) = \bsI^{\alpha,\beta}(\theta^*,K) + p_{\D}^* N I(\theta^*)
}
for any $\alpha > 0$ and any $\beta \geq 0$. This example gives an upper bound for the increase of the Fisher information when the number of objects is increased, since it depicts the case where there is no data association uncertainty for these objects. This would correspond in practice to a case where the added objects are in an area where there is no false alarm and where these objects are ``far'' from the existing objects as well as ``far'' from each other, where ``far'' depends on the likelihood.
\end{example}

\section{Conclusion}
\label{sec:conclusion}

The first important result in this article is the proof of consistency of the maximum likelihood estimator for multi-target tracking under \emph{weak} conditions, where weak means that these conditions are as often as possible applying to the single-target dynamics and observation. Asymptotic normality holds under additional assumptions and the second part of the article brings understanding to the asymptotic variance of the maximum likelihood estimate by analysing the Fisher information matrix corresponding to multi-target tracking. Qualitative results are obtained in the general case, that is, the Fisher information decreases with data association uncertainty, detection failures and in the presence of false alarms. Quantitative results are also derived in important special cases:
\begin{enumerate*}[label=\alph*)]
\item one static target with false alarm and unknown data association,
\item multiple static targets with unknown data association under a particular observation model, and
\item multiple targets with detection failures.
\end{enumerate*}

Future works include the study of identifiability of specific observation-to-track associations, instead of marginalising over all possibilities as considered in this article. Such an approach involves additional challenges since the parameters to be learned increase in dimensionality with time, so that it is not a special case of the results presented here.

\section*{Acknowledgement}

All authors were supported by the Singapore ministry of education tier 1 grant number R-155-000-182-114.

\appendix

\section{Assumptions for \texorpdfstring{\cref{thm:consistencyMLE}}{Theorem~\ref{thm:consistencyMLE}}}
\label{sec:assumptionAsymptoticNormality}
The following assumptions are required for the proof of the asymptotic normality of the maximum likelihood estimator in multi-target tracking. The norm $\|\cdot\|$ is defined as $\|M\| = \sum_{i,j} |M_{i,j}|$ for any matrix $M$.
\begin{enumerate}[hyp, resume]
\item \label{it:asymptoticNormality1} For all $K \in S^{\T}$, all $\bsx,\bsx' \in \bbX^K$ and all $\bsy \in \bbY^{\times}$, the mappings $\bstheta \mapsto \bsf_{\bstheta}(\bsx \given \bsx')$ and $\bstheta \mapsto \bsg_{\bstheta}(\bsy \given \bsx)$ are twice continuously differentiable on the hyperplane of $\bsTheta$ made of parameters with a number of target equal to $K$.
\item \label{it:asymptoticNormality2} It holds that
\eqnsa{
\sup_{\bstheta \in \bsTheta} \sup_{\bsx,\bsx' \in \bbX^K} \| \nabla_{\bstheta} \log \bsf_{\bstheta}(\bsx \given \bsx') \| < \infty \quad\text{and}\quad \sup_{\bstheta \in \bsTheta} \sup_{\bsx,\bsx' \in \bbX^K} \| \nabla^2_{\bstheta} \log \bsf_{\bstheta}(\bsx \given \bsx') \| < \infty
}
and that
\eqnsa{
\bar\bbE_{\bstheta^*}\Big[\sup_{\bstheta \in \bsTheta} \sup_{\bsx \in \bbX^K} \| \nabla_{\bstheta} \log \bsg_{\bstheta}(\bsY \given \bsx) \| \Big] & < \infty \\
\bar\bbE_{\bstheta^*}\Big[\sup_{\bstheta \in \bsTheta} \sup_{\bsx \in \bbX^K} \| \nabla^2_{\bstheta} \log \bsg_{\bstheta}(\bsY \given \bsx) \| \Big] & < \infty.
}
\item \label{it:asymptoticNormality3} For all $\bsy \in \bbY^{\times}$, there exists an integrable function $h_{\bsy} : \bigcup_{k\geq 0} \bbX^k \to \bbR^+$ such that $\sup_{\bstheta \in \bsTheta} \bsg_{\bstheta}(\bsy \given \bsx) \leq h_{\bsy}(\bsx)$. For all $\bstheta \in \bsTheta$ and for all $\bsx \in \bbX^K$, there exist integrable functions $h^1_{\bsx},h^2_{\bsx} : \bbY^{\times} \to \bbR^+$ such that
\eqns{
\| \nabla_{\bstheta}\, \bsg_{\bstheta}(\bsy \given \bsx) \| \leq h^1_{\bsx}(\bsy) \AND \| \nabla^2_{\bstheta}\, \bsg_{\bstheta}(\bsy \given \bsx) \| \leq h^2_{\bsx}(\bsy).
}
\end{enumerate}

\section{Proof of \texorpdfstring{\cref{lem:boundedness}}{Lemma~\ref{lem:boundedness}}}
\label{proof:lem:boundedness}
It follows from Assumption~\ref{it:boundIntLikelihood} that the supremum $\bsb^{\C}_+$ of the clutter density $\bsp_{(\lambda,\psi)}$ characterised for any $k \in \bbN_0$ and any $\bsy \in \bbY^k$ by $\bsp_{(\lambda,\psi)}(\bsy) = \Po_{\lambda}(k)\prod_{i=1}^k p_{\psi}(\bsy_i)$ verifies $\bsb^{\C}_+ < \infty$ since
\eqnsa{
\sup_{\lambda \in S^{\C}} \bigg( \sum_{k \geq 0} \sup_{(\psi,\bsy) \in \Psi \times \bbY^k} \bsp_{(\lambda,\psi)}(\bsy) \bigg) & = \sup_{\lambda \in S^{\C}} \sum_{k \geq 0} \dfrac{(\lambda b^{\C}_+)^k e^{-\lambda}}{k!} = \sup_{\lambda \in S^{\C}} e^{\lambda(b^{\C}_+ -1)} < \infty,
}
and since all the terms in the sum are positive. It then holds that
\eqnsa{
\hat\bsb_+ & \leq \sup_{(p_{\D},K,\bsy) \in (0,1) \times S^{\T} \times \bbY^{\times}} \bsb^{\C}_+ \sum_{i = 0}^{K \land \#\bsy} \binom{K}{i} (p_{\D}b^{\T}_+)^i (1-p_{\D})^{K-i} \\
& \leq \sup_{(p_{\D},K) \in (0,1) \times S^{\T}} \bsb^{\C}_+ (1-p_{\D}+p_{\D}b^{\T}_+)^K < \infty,
}
which concludes the first part of the proof. For any $k \in \bbN_0$ and any $\bsy \in \bbY^k$
\eqnsa{
\int \bsg_{\bstheta}(\bsy \given \bsx) \d\bsx
& \geq \sum_{\substack{\bsd \in \{0,1\}^K \\ |\bsd| \leq k}} \prod_{i=1}^k \big[b_-^{\T}(\bsy_i) \land b_-^{\C}(\bsy_i) \big] \Po_{\lambda}(k - |\bsd|) q_{\bstheta}(\bsd) \\
& = \sum_{d=0}^{K\wedge k} \prod_{i=1}^k \big[ b_-^{\T}(\bsy_i) \land b_-^{\C}(\bsy_i) \big] \Po_{\lambda}(k - d) \Bi_{p_{\D}}^K (d) \\
& = \Bi_{p_{\D}}^K * \Po_{\lambda}(k) \prod_{i=1}^k \big[ b_-^{\T}(\bsy_i) \land b_-^{\C}(\bsy_i) \big] \\
& \geq \inf_{\bstheta \in \bsTheta}\Bi_{p_{\D}}^K * \Po_{\lambda}(k) \prod_{i=1}^k \big[ b_-^{\T}(\bsy_i) \land b_-^{\C}(\bsy_i) \big].
}
It also holds that $\inf_{\bstheta \in \bsTheta} \Bi_{p_{\D}}^K * \Po_{\lambda}(k) > 0$ for any $k \in \bbN_0$ since the support of $\Po_{\lambda}$ is $\bbN_0$ for any $\lambda \in S^{\C}$ which guarantees that the convolution has also $\bbN_0$ as a support so that the infimum is strictly greater than zero. It follows that $\bsb_-(\bsy) > 0$ and, considering \eqref{eq:logInfConvBiPo}, that $\bar\bbE_{\bstheta^*}[| \log \bsb_-(\bsY)|] < \infty$. Similarly, for any $k \in \bbN_0$ and any $\bsy \in \bbY^k$
\eqns{
\bsb_+(\bsy) \leq \sup_{\bstheta \in \bsTheta} \Bi_{p_{\D}}^K * \Po_{\lambda}(k) \prod_{i=1}^k \big[ b_+^{\T}(\bsy_i) \lor b_+^{\C}(\bsy_i) \big],
}
which is finite when $k$ is finite. In the infinite case, noticing that $\Po_\lambda(k-K) p^K_{\D}$ is the leading term in the convolution, we find that
\eqnsa{
\lim_{k \to \infty} \sup_{\bsy \in \bbY^k} \bsb_+(\bsy) & \leq \sup_{\bstheta \in \bsTheta} \lim_{k \to \infty} K p^K_{\D} \Po_\lambda(k-K) \Big[ \sup_{y\in \bbY} \big(b_+^{\T}(y) \lor b_+^{\C}(y)\big) \Big]^k \\
& \leq \sup_{\bstheta \in \bsTheta} c_{\bstheta} \lim_{k \to \infty} \dfrac{e^{-\lambda}}{(k-K)!} \Big[ \sup_{y\in \bbY} \big(b_+^{\T}(y) \lor b_+^{\C}(y)\big) \Big]^{k-K} < \infty
}
where $c_{\bstheta}$ is a finite constant, which concludes the proof of the lemma.

\section{Proof of \texorpdfstring{\cref{thm:transferIdentifiability}}{Theorem~\ref{thm:transferIdentifiability}}}
\label{proof:thm:transferIdentifiability}
The two cases of \cref{thm:transferIdentifiability} are proved separately as follows:
\begin{enumerate}[label=\alph*),wide]
\item When $\bstheta \in \bsTheta_{\lambda=0}$, the joint probability of the observations when the system is initialised with its stationary distribution is characterised by
\eqns{
\bar\bsP_{\bstheta}(B) = \int \ind{B}(\bsy_{1:n}) \prod_{t = 1}^n \bsg_{\bstheta}(\bsy_t \given \bsx_t) 
\prod_{i=1}^K \bigg[ \pi_{\theta}(\bsx_{0,i}) \prod_{t = 1}^n f_{\theta}(\bsx_{t,i} \given \bsx_{t-1,i}) \bigg]\d \bsy_{1:n} \d \bsx_{0:n},
}
for any measurable subset $B = B_1 \times \dots \times B_n$ of $(\bbY^{\times})^n$ with
\eqns{
\bsg_{\bstheta}(\bsy \given \bsx) \defeq \sum_{\substack{\bsd \in \{0,1\}^K \\ |\bsd| = m}} \bigg[ \sum_{\sigma \in \Sym(m)} \prod_{i = 1}^{m} g_{\theta}\big(\bsy_{\sigma(i)} \given \bsx_{r(i)}\big) u_m(\sigma) \bsq_{\bstheta}(\bsd) \bigg].
}
for any $m \in \bbN_0$ and any $(\bsx,\bsy) \in \bbX^K \times \bbY^m$. Assuming that $B_t$ is a measurable subset of $\bbY^{K}$ of the form $A_t \times \dots \times A_t$ for any $1 \leq t \leq n$, then the sum over $\bsd$ collapses to a single term where all targets are detected and all the terms in the sum over $\sigma$ are equal, so that
\eqnsa{
w^1_{\bstheta} & \defeq \bar\bsP_{\bstheta}(B) \\
& = p_{\D}^{Kn} \int \prod_{i = 1}^K \bigg[ \pi_{\theta}(\bsx_0) \prod_{t = 1}^n \Big[ \ind{A_t}(\bsy_{t,i}) g_{\theta}\big(\bsy_{t,i} \given \bsx_{t,i}\big) f_{\theta}(\bsx_{t,i} \given \bsx_{t-1,i}) \Big] \bigg]\d \bsy_{1:n} \d \bsx_{0:n}.
}
A second case that can be considered is when $B_t$ represents the configuration where there are $m \leq K$ observations without considering their locations for all $1 \leq t \leq n$, i.e.\ $B_t = \bbY \times \dots \times \bbY$, in which case it holds that
\eqns{
w^{2,m}_{\bstheta} \defeq \bar\bsP_{\bstheta}(B) = \big( \Bi_{p_{\D}}^K(m) \big)^n
}
If $(K,p_{\D}) \neq (K^*,p^*_{\D})$ then we can show that $w^{2,K}_{\bstheta} = w^{2,K}_{\bstheta^*}$ and $w^{2,K-1}_{\bstheta} = w^{2,K-1}_{\bstheta^*}$ cannot hold at the same time for any $\bstheta \in \bsTheta_{\lambda=0}$. Alternatively, if $(K,p_{\D}) = (K^*,p^*_{\D})$ then $w^1_{\bstheta} \neq w^1_{\bstheta^*}$ follows easily from the identifiability of $\theta^*$. These two cases considered together show that the distributions associated to $\bstheta$ and $\bstheta^*$ differ in some subset of the multi-target observation space so that $\bar\bsP_{\bstheta} \neq \bar\bsP_{\bstheta^*}$.
\item When $\bstheta \in \bsTheta|_{K=1}$, the multi-target likelihood becomes
\eqns{
\bsg_{\bstheta}(\bsy \given x) =  (1-p_{\D}) \Po_{\lambda}(m) \prod_{i = 1}^m p_{\psi^*}(\bsy_i) + \dfrac{p_{\D}}{m}\sum_{i=1}^m g_{\theta^*}\big(\bsy_i \given x\big) \Po_{\lambda}(m-1) \prod_{\substack{1\leq j \leq m \\ j \neq i}} p_{\psi^*}(\bsy_j),
}
for any $m \in \bbN_0$ and any $(x,\bsy) \in \bbX \times \bbY^m$. 
Marginalising over the location of the observations at each time step and considering the case where there are $m$ observations, i.e.\ $B_t = \bbY \times \dots \times \bbY$, gives
\eqns{
w^m_{\bstheta} \defeq \bar\bsP_{\bstheta}(B) = (1 - p_{\D})\Po_{\lambda}(m) + p_{\D} \Po_{\lambda}(m-1), \qquad m \geq 1
}
and $w^0_{\bstheta} \defeq ( 1 - p_{\D}) e^{-\lambda}$. Assuming that $\bstheta \neq \bstheta^*$ and considering that $(p_{\D},\lambda) \neq (p_{\D}^*,\lambda^*)$ it follows that $w^0_{\bstheta} = w^0_{\bstheta^*}$, $w^1_{\bstheta} = w^1_{\bstheta^*}$ and $w^2_{\bstheta} = w^2_{\bstheta^*}$ cannot all hold at the same time, 
which concludes the proof. 
\end{enumerate}

\section{Proof of \texorpdfstring{\cref{thm:informationLoss}}{Theorem~\ref{thm:informationLoss}}}
\label{proof:thm:informationLoss}

\begin{lemma}
\label{lem:informationLoss}
For given integers $m$ and $K$, let $\bsP_{\theta}$ be a family of probability measures on $\bbY^{mK}$ indexed by $\theta \in \Theta$ and let $\bsp_{\theta}$ denote the corresponding probability density w.r.t.\ a common reference measure, for all $\theta$, on $\bbY^{mK}$. Assume that $\bsp_{\theta}\left(\bsy_{1},\ldots,\bsy_{m}\right)>0$ for any $\theta$ and $\left(\bsy_{1},\ldots,\bsy_{m}\right)$.

 For any integers $\alpha \geq 1$ and $\beta \geq 0$, let the random vectors $\left(\bsY'_{1},\ldots,\bsY'_{m}\right)$ be conditionally
independent given $\left(\bsY_{1},\ldots,\bsY_{m}\right)$, with law $\bsP^{\alpha,\beta}_{\bsY'_{1:m}\vert\bsY_{1:m}}=\bsP^{\alpha,\beta}_{\bsY'_{1}\vert\bsY_{1}}\ldots\bsP^{\alpha,\beta}_{\bsY'_{m}\vert\bsY_{m}}$
and each $\bsP^{\alpha,\beta}_{\bsY'_{i}\vert\bsY_{i}}$ is defined
as in (\ref{eq:perturbedModel}) via a process of thinning, augmentation with clutter with density $p_{\psi}$ on $\bbY$ and random permutation. Assume $p_{\psi} >0$.
\begin{enumerate}
\item Consider any $\theta$ and $(\alpha,\beta)$ such that $\alpha > 1$ or $\beta > 0$. If $f(\bsY_1,\ldots,\bsY_m)=\bbE[ f(\bsY_1,\ldots,\bsY_m) | \bsY'_{1},\ldots,\bsY'_{m} ]$ then $f(\bsY_1,\ldots,\bsY_m)$ is constant almost surely.
\item  Let the probability measure of $\left ( \bsY'_{1},\ldots,\bsY'_{m}\right)$ be $\bsP^{\alpha,\beta}_{\theta}$ and its corresponding probability density be $\bsp^{\alpha,\beta}_{\theta}$. Assume that the densities $\bsp_{\theta}$ and $\bsp^{\alpha,\beta}_{\theta}$ 
are differentiable w.r.t.\ $\theta$  then
\begin{multline}
\label{eq:lem:informationLoss}
\bbE [ \nabla_{\theta} \log \bsp_{\theta}(\bsY_1,\ldots,\bsY_m) \cdot \nabla_{\theta} \log \bsp_{\theta}(\bsY_1,\ldots,\bsY_m)^{\tr} ] \\
\geq \bbE[ \nabla_{\theta} \log \bsp^{\alpha,\beta}_{\theta}(\bsY'_1,\ldots,\bsY'_m) \cdot \nabla_{\theta} \log \bsp^{\alpha,\beta}_{\theta}(\bsY'_1,\ldots,\bsY'_m)^{\tr} ],
\end{multline}
with the inequality being strict if and only if $\alpha > 1$ or $\beta > 0$ and if the l.h.s.\ is strictly greater than $0$.
\end{enumerate}
\end{lemma}

\begin{proof}[Proof of Part 1 of \cref{lem:informationLoss}]
When $\alpha = 1$ and $\beta>0$, which  corresponds to no random permutation (or no observation association uncertainty) but only random thinning, the result follows from \cref{lem:multidel} and \cref{rem:strongerHypMultiDel}.
When $\alpha > 1$ and $\beta=0$, which  corresponds to no random thinning but only random permutation, the result follows from \cref{cor:multperm_vec}. When $\alpha \geq 1$ and $\beta>0$, i.e.\ both random thinning and random permutation are present, the result of \cref{lem:multidel} based on $\bsY_D$, i.e.\ for $(\alpha=1,\beta>0)$, can be extended to the random variable $\bsY'_D = S_{\varsigma}(\bsY_D \oplus \hat\bsY)$ as follows:
\begin{enumerate}[wide]
\item \label{it:permMakesCoarse} By noticing that $\sigma(\bsY'_D) \subseteq \sigma(\varsigma, \bsY_D, \hat\bsY)$ (i.e.\ the $\sigma$-algebra generated by $\bsY'_D$ is coarser than the one generated by $\bar\bsY_D \defeq (\varsigma, \bsY_D, \hat\bsY)$). Indeed, for any $C = B_1 \times B_2 \times \dots \in \calB(\bbY^{\times})$ with all but finitely many $B_i$ equal to $\bbY$, it holds that
\eqns{
\bsY'^{-1}_D(C) = \bigcup_{\sigma} \bar\bsY_D^{-1}(\{\sigma\} \times B_{\sigma(1)} \times B_{\sigma(2)} \times \dots) \in \sigma(\bar\bsY_D)
}
where the inclusion follows by countable union. The family of subsets of the same form as $C$ is a generating family of $\calB(\bbY^{\times})$ from which the result follows. The result extends straightforwardly to any collection of $m$ thinned observation vectors being independently perturbed by false alarm and permutation since the corresponding $\sigma$-algebra is generated by the set of rectangles of the form $C_1 \times \dots \times C_m$ with $C_i \in \calB(\bbY^{\times})$.
\item As a consequence of \ref{it:permMakesCoarse}, the fact that $f(\bsY) = \bbE[f(\bsY) \given \bsY'_D]$ a.s.\ implies that
\eqns{
\bbE[f(\bsY) \given \bar\bsY_D] = \bbE\big[ \bbE[f(\bsY) \given \bsY'_D] \given \bar\bsY_D \big] = \bbE[f(\bsY) \given \bsY'_D] \quad a.s.,
}
so that, since $\hat\bsY$, $\varsigma$ and $\bsY_D$ are independent, it holds that
\eqns{
f(\bsY) = \bbE[f(\bsY) \given \bar\bsY_D] = \bbE[f(\bsY) \given \bsY_D] \quad a.s.,
}
which under the assumption of \cref{lem:multidel} implies that $f(\bsY)$ is constant almost surely. As in \ref{it:permMakesCoarse}, the same result holds for collection of $m$ observation vectors.
\end{enumerate}
\end{proof}

\begin{proof}[Proof of Part 2 of \cref{lem:informationLoss}]

Let $\bsY$ be the $K$ measurements of $K$ targets, $\hat{\bsY}$ be the clutter, $\varsigma$ the random permutation, $\bsD$ the $K$ dimensional vector of deletions. Let $\bsY'=S_{\varsigma}((R_{\bsD}\bsY)\oplus\hat{\bsY})$. The missing target generated observations are $\bsY_{\m} =R_{\bsD'}Y$ where $\bsD'=\one-\bsD$.

Let $\bsp_{\bsY',\varsigma,\bsD,\bsY_{\m}}(\bsy',\sigma,\bsd,\bsy_{m})$ denote the joint pdf/pmf of $(\bsY',\bsY_{\m},\bsD,\varsigma)$ that depends implicitly on $\theta$. Using the change of variable formula, noting that $(\bsY,\hat{\bsY})=F(\bsY',\bsY_{\m},\bsD,\varsigma)$ where the mapping $F$ is a permutation of $(\bsY',\bsY_{\m})$ and hence the Jacobian of the transformation has determinant~$1$, it follows that
\eqns{
\bsp_{\bsY',\bsY_{\m},\bsD,\varsigma}(\bsy',\bsy_{\m},\bsd,\sigma) = \bsp_{\bsY,\hat{\bsY},\bsD,\varsigma}(F(\bsy',\bsy_{\m},\bsd,\sigma),\bsd,\sigma)
}
for $\bsy'=S_{\sigma}((R_{\bsd}\bsy)\oplus\hat{\bsy})$, $\bsy_{\m} = R_{\bsd'}\bsy$ where $\bsd' = \one-d$. Since it holds that $\nabla_{\theta}\log \bsp_{\bsY,\hat{\bsY},\bsD,\varsigma}(\bsy,\hat{\bsy},\bsd,\sigma) = \nabla_{\theta}\log \bsp_{\bsY}(\bsy)$, we deduce that
\eqns{
\nabla_{\theta} \log \bsp_{\bsY',\bsY_{\m},\bsD,\varsigma}(\bsy',\bsy_{\m},\bsd,\sigma) = \nabla_{\theta} \log \bsp_{\bsY}(F_{\T}(\bsy',\bsy_{\m},\bsd,\sigma))
}
where $F_{\T}(\bsy',\bsy_{\m},\bsd,\sigma)$ is the projection of $F(\bsy',\bsy_{\m},\bsd,\sigma)$ on the coordinates describing~$\bsY$.
Now it follows that, almost surely,
\eqns{
\nabla_{\theta} \log \bsp_{\bsY',\bsY_{\m},\bsD,\varsigma}(\bsY',\bsY_{\m},\bsD,\varsigma) = \nabla_{\theta} \log \bsp_{\bsY}(\bsY).
}
Let $\bsp_{\bsY'}$ denote the density of $\bsY'$. Then 
\eqnl{eq:lemma7Part2Fisher}{
\nabla_{\theta} \log \bsp_{\bsY'}(\bsY') = \bbE_{\bsP_{\theta}}[ \nabla_{\theta} \log \bsp_{\bsY}(\bsY)\mid \bsY' ].
}
Applying \eqref{eq:lemma7Part2Fisher} to the joint random variables $\bsY_{1:m}$ and $\bsY'_{1:m}$ defined in the lemma, it follows that
\eqns{
\nabla_{\theta} \log \bsp^{\alpha,\beta}_{\theta}(\bsY'_{1:m}) = \bbE_{\bsP_{\theta}}\big[ \nabla_{\theta} \log \bsp_{\theta}(\bsY_{1:m}) \given \bsY'_{1:m} \big].
}
Let $v \in \bbR^{d_{\Theta}}$, then Jensen's inequality applied to the function $x \mapsto x^2$ and to the random variable $v^{\tr} \nabla_{\theta} \log \bsp_{\theta}(\bsY_{1:m})$ yields
\eqns{
\bbE_{\bsP_{\theta}} \big[ v^{\tr} \nabla_{\theta} \log \bsp_{\theta}(\bsY_{1:m}) \given \bsY'_{1:m} \big]^2 \leq \bbE_{\bsP_{\theta}}\big[ \big(v^{\tr} \nabla_{\theta} \log \bsp_{\theta}(\bsY_{1:m}) \big)^2 \given \bsY'_{1:m} \big]
}
almost surely, so that
\eqnsa{ 
v^{\tr} \bbE_{\bsP^{\alpha,\beta}_{\theta}}[ \nabla_{\theta} \log \bsp^{\alpha,\beta}_{\theta}(\bsY_{1:m}) \cdot \nabla_{\theta} & \log \bsp^{\alpha,\beta}_{\theta}(\bsY_{1:m})^{\tr} ]v \\
& = \bbE_{\bsP^{\alpha,\beta}_{\theta}} \big[ \bbE_{\bsP_{\theta}} \big[ v^{\tr} \nabla_{\theta} \log \bsp_{\theta}(\bsY_{1:m}) \given \bsY'_{1:m} \big]^2 \big] \\
& \leq v^{\tr} \bbE_{\bsP_{\theta}}[ \nabla_{\theta} \log \bsp_{\theta}(\bsY_{1:m}) \cdot \nabla_{\theta} \log \bsp_{\theta}(\bsY_{1:m})^{\tr} ] v
}
which proves \eqref{eq:lem:informationLoss}. Since Jensen's inequality has been applied to a strictly convex function, the case of equality:
\eqnsml{
v^{\tr} \bbE_{\bsP^{\alpha,\beta}_{\theta}}[ \nabla_{\theta} \log \bsp^{\alpha,\beta}_{\theta}(\bsY_{1:m}) \cdot \nabla_{\theta} \log \bsp^{\alpha,\beta}_{\theta}(\bsY_{1:m})^{\tr} ]v \\
= v^{\tr} \bbE_{\bsP_{\theta}}[ \nabla_{\theta} \log \bsp_{\theta}(\bsY_{1:m}) \cdot \nabla_{\theta} \log \bsp_{\theta}(\bsY_{1:m})^{\tr} ] v,
}
holds if and only if, for all $v \in \bbR^{d_{\Theta}}$, $v^{\tr} \nabla_{\theta} \log \bsp_{\theta}(\bsY_{1:m}) = \bbE[ v^{\tr} \nabla_{\theta} \log \bsp_{\theta}(\bsY_{1:m}) \given \bsY'_{1:m} ]$ is $\sigma(\bsY'_{1:m})$-measurable. Part 1 of the lemma shows that $\nabla_{\theta} \log \bsp_{\theta}(\bsY_{1:m})$ is $\sigma(\bsY'_{1:m})$-measurable if and only if it is constant a.s.. Given that $\bbE_{\bsP_{\theta}}[\nabla_{\theta} \log \bsp_{\theta}(\bsY_{1:m})] = 0$ it follows that the function itself is equal to $0$ since it is constant, hence proving the lemma.
\end{proof}

The observation model \eqref{eq:perturbedModel} does not imply the equality of the gradients of $\log \bsp_{\bstheta}(\bsY', \bsY)$ and $\log \bsp_{\bstheta}(\bsY)$ w.r.t.\ $\bstheta$ since $\bsP_{\theta}(\d\bsY' \given \bsY)$ depends on $p_{\D}$, $\lambda$ and $\psi$ which are parameters included in $\bstheta$. The interest is however in the information loss w.r.t.\ $\theta$ so that the result of \cref{lem:informationLoss} is satisfying.

\begin{proof}[Proof of \cref{thm:informationLoss}]
The considered perturbed observation model has the same properties as the one studied in \cite{Dean2014}, i.e.\ that $\nabla_{\theta} \log \bsp^{\alpha,\beta}_{\theta}(\bsY, \bsY^{\alpha,\beta}) = \nabla_{\theta} \log \bsp_{\theta}(\bsY)$ holds almost surely. The result of \cite[Lemma~3]{Dean2014} and \cite[Remark~9]{Dean2014} can therefore be used directly in the context of interest to give
\eqns{
\bsI^{\alpha,\beta}_{\loss}(\theta^*) = \bar\bbE_{\theta^*}\Big[ \bsI^{(m)}_{\bsY_{-\infty:-1},\bsY^{\alpha,\beta}_{m:\infty}}(\theta^*) \Big],
}
where, for any $\bsY_{-\infty:-1}$, any $\bsY^{\alpha,\beta}_{m:\infty}$ and any integer $m \geq 1$,
\eqnsa{
\bsI^{(m)}_{\bsY_{\infty:-1},\bsY^{\alpha,\beta}_{m:\infty}}(\theta^*) = & \dfrac{1}{m} \bar\bbE_{\theta^*}\Big[ \nabla_{\theta} \log \bsp_{\theta^*}\big(\bsY_{0:m-1} \given \bsY_{-\infty:-1},\bsY^{\alpha,\beta}_{m:\infty}\big) \cdot \\
& \qquad\qquad \nabla_{\theta} \log \bsp_{\theta^*}\big(\bsY_{0:m-1} \given \bsY_{-\infty:-1},\bsY^{\alpha,\beta}_{m:\infty}\big)^{\tr} \Given \bsY_{-\infty:-1},\bsY^{\alpha,\beta}_{m:\infty} \Big] \\
& -\dfrac{1}{m} \bar\bbE_{\theta^*}\Big[ \nabla_{\theta} \log \bsp_{\theta^*}\big(\bsY_{0:m-1} \given \bsY_{-\infty:-1},\bsY^{\alpha,\beta}_{m:\infty}\big) \cdot \\
& \qquad\qquad \nabla_{\theta} \log \bsp_{\theta^*}\big(\bsY_{0:m-1} \given \bsY_{-\infty:-1},\bsY^{\alpha,\beta}_{m:\infty}\big)^{\tr} \Given \bsY_{-\infty:-1},\bsY^{\alpha,\beta}_{m:\infty} \Big].
}
The objective is to prove that
\eqnl{eq:condProofInfLoss}{
\bar\bbE_{\theta^*}\Big[\bsI^{(m)}_{\bsY_{-\infty:-1},\bsY^{\alpha,\beta}_{m:\infty}}(\theta^*)\Big] = 0
}
for all $m\geq1$ implies that $I(\theta^*) = 0$. From \cref{lem:informationLoss} applied to the involved conditional laws, \eqref{eq:condProofInfLoss} implies that $\nabla_{\theta} \log \bsp_{\theta^*}(\bsY_{0:m-1} \given \bsY_{-\infty:-1},\bsY^{\alpha,\beta}_{m:\infty}) = 0$ almost surely for almost all $\bsY_{-\infty:-1}$ and almost all $\bsY^{\alpha,\beta}_{m:\infty}$. Following the same principle as in \cite[Lemma~4]{Dean2014}, it follows that if \eqref{eq:condProofInfLoss} holds for all $m \geq 1$ then $\nabla_{\theta} \log \bsp_{\theta^*}(\bsY_0 \given \bsY_{-\infty:-1}) = 0$ almost surely, which in turn implies that $I(\theta^*)=0$.
\end{proof}

\subsection{Supporting results for the proof of Part 1 of Lemma \ref{lem:informationLoss}} \label{subsec_appendix_informationLoss} 
\hfill

\subsubsection{Supporting results for the proof of \cref{lem:informationLoss} for permutation uncertainty but no deletion, i.e. $(\alpha>1, \beta=0)$}

\begin{lemma}[Randomly permuting a random vector.]
\label{lem:multpermute}
Let $\bsY=(Y_{1},\ldots,Y_{n})$ and $\hat{\bsY}=(Y_{n+1},\ldots,Y_{n+m})$. Let $\varsigma$ denote the randomised permutation which is independent of $(\bsY,\hat{\bsY})$ and let $\mathbf{Z}=(Y_{\varsigma(1)},\ldots,Y_{\varsigma(m+n)})$.  Assume $\varsigma$ permits, at the least, the exchange of any two indices, i.e. 
\[
\mathbb{P} (\varsigma(i)=j,\varsigma(j)=i, \{ \varsigma(k)=k:k\neq i,j \}  )>0
\]
for all $i,j$. Furthermore, $\mathbb{P} (\varsigma=(1,\ldots,n) )>0$. Assume the law of $(\bsY,\hat{\bsY})$ satisfies $\nu^{n+m}\gg\mathbb{P}_{\bsY,\hat{\bsY}}\gg\nu^{n+m}$ where $\nu$ is some probability measure and $\nu^{n+m}$ the product probability measure on $\mathbb{Y}^{n+m}$. 

If $f(\bsY)=\mathbb{E} [ f(\bsY)\vert\mathbf{Z} ]$, then $f(\bsY)$ is a constant almost surely.
\end{lemma}

\begin{proof}[Proof of \cref{lem:multpermute}] The proof is completed for the case $m=1$ and easily generalised to $m>1$. Likewise, to present the arguments we employ in the clearest way, we consider the case $(n=2,m=1)$. (The case $n+m=2$ sheds too little light to make the generalisation apparent.) 

Let $g(\mathbf{Z})=\mathbb{E} [f(\bsY)\vert\mathbf{Z} ]$. For any $\sigma$ such that $\mathbb{P} (\varsigma=\sigma )>0$,
\[
\mathbb{E} [ |f(\bsY)-g(Y_{\sigma(1)},\ldots,Y_{\sigma(n+1)}) |\mathbb{I}_{[\varsigma=\sigma]} ] =0.
\]
Since $\varsigma$ is independent of $(\bsY,\hat{\bsY})$, we have 
\[
\mathbb{E} [ |f(\bsY)-g(Y_{\sigma(1)},\ldots,Y_{\sigma(n+1)}) | ] =0
\]
and thus $g(Y_{\sigma(1)},\ldots,Y_{\sigma(n+1)})=g(Y_{\sigma'(1)},\ldots,Y_{\sigma'(n+1)})$ almost surely for any other $\sigma'$ such that $\mathbb{P} (\varsigma=\sigma' )>0$. 

Now consider the case $(n=3,m=1)$. The preceding statements imply, $\mathbb{P}_{\bsY,\overline{\bsY}}$ almost everywhere,
\begin{align}
g(Y_{1},Y_{2},Y_{3}) & =f(Y_{1},Y_{2})\label{eq:proof_multpermute_eq1}\\
g(Y_{1},Y_{2},Y_{3}) & =g(Y_{3},Y_{2},Y_{1})\label{eq:proof_multpermute_eq2}\\
g(Y_{1},Y_{2},Y_{3}) & =g(Y_{1},Y_{3},Y_{2}).\label{eq:proof_multpermute_eq2-1}
\end{align}
Indeed these also hold $\nu^{3}$ almost every. (Due to the assumption of mutual absolutely continuity, statements holds $\mathbb{P}_{\bsY,\overline{\bsY}}$ almost everywhere if and only if they hold $\nu^{3}$ almost every.) We will show that the further implication 
\begin{equation}
f(Y_{1},Y_{2})=f(Y_{3},Y_{2})=f(Y_{1},Y_{3})\label{eq:proof_multpermute_eq2-2}
\end{equation}
holds $\nu^{3}$ almost everywhere may be derived. Once this is done, to complete the proof, we will further manipulate (\ref{eq:proof_multpermute_eq2-2}) under the assumption that the random variables $Y_{i}$ are independently and identically distributed with respect to measure $\nu$ to show that $f=c$, for some constant $c$, $\nu^{3}$ almost everywhere.

From the first equality of (\ref{eq:proof_multpermute_eq2-2}),
\eqns{
f(Y_{1},Y_{2})=\mathbb{E}_{\nu^{3}} ( f(Y_{3},Y_{2}) |Y_{1},Y_{2} )=\mathbb{E}_{\nu^{3}} ( f(Y_{3},Y_{2}) |Y_{2} )=h(Y_{2})
}
for some function $h$. That is $f(Y_{1},Y_{2})$ collapses to a function of variable $Y_{2}$ only, which is denoted by $h(Y_{2})$. Using the second equality of (\ref{eq:proof_multpermute_eq2-2}), $h(Y_{2})=f(Y_{1},Y_{3})$ and thus it must be that $h$ is a constant as $Y_{i}$ are independent.

We now verify (\ref{eq:proof_multpermute_eq1})-(\ref{eq:proof_multpermute_eq2}) implies $f(Y_{1},Y_{2})=f(Y_{3},Y_{2})$ of (\ref{eq:proof_multpermute_eq2-2}).  We have 
\[
\mathbb{E}_{\nu^{3}} [ |f(Y_{1},Y_{2})-g(Y_{3},Y_{2},Y_{1}) | ] = 0
\]
and a change of variable gives $\mathbb{E}_{\nu^{3}} [ |f(Y_{3},Y_{2})-g(Y_{1},Y_{2},Y_{3}) | ] =0$.  The same procedure applied to (\ref{eq:proof_multpermute_eq1})-(\ref{eq:proof_multpermute_eq2-1}) shows the second equality of (\ref{eq:proof_multpermute_eq2-2}).
\end{proof}

\Cref{cor:multperm_clutter} extends \cref{lem:multpermute} to the situation when $\hat{\bsY}$ therein follows the law of a clutter process as defined in Section \ref{ssec:atlObservationModel} (see \ref{eq:perturbedModel}.)

\begin{corollary}
\label{cor:multperm_clutter} Let $(\hat{Y}_{1},\hat{Y}_{2},\ldots)$ be an infinite sequence of independent $\mathbb{Y}$-valued random variables with $\hat{Y}_{i}\sim P_{\psi}$. Let $\bsY=(Y_{1},\ldots,Y_{K})$ be a vector of $\mathbb{Y}$-valued random variables which is independent of $(\hat{Y}_{1},\hat{Y}_{2},\ldots)$. Let $\hat{M}\in\mathbb{N}_{0}$ be non-negative random variable independent of $ (\bsY,\hat{Y}_{1:\infty} )$.  Let $\mathbf{Z}=S_{\varsigma} (\bsY\oplus\hat{\bsY} )$ where $\hat{\bsY}=\hat{Y}_{1:\hat{M}}$ and $S_{\varsigma}$ is the random permutation matrix defined as in (\ref{eq:perturbedModel}) of Section \ref{ssec:atlObservationModel}.  Assume $ (P_{\psi} )^{K}\gg\mathbb{P}_{\bsY}\gg (P_{\psi} )^{K}$. If $f(\bsY)=\mathbb{E} [ f(\bsY)\vert\mathbf{Z} ]$, then $f(\bsY)$ is a constant almost surely.
\end{corollary}

\begin{proof}
Let $g(\mathbf{Z})=\mathbb{E}[ f(\bsY)\vert\mathbf{Z} ]$ then $\mathbb{E}[ |f(\bsY)-g(\mathbf{Z}) | |\hat{M}=m ] = 0$ for all $m$ such that $\mathbb{P} (\hat{M}=m )>0$. Since $\hat{M}$ is independent of $ (\bsY,\hat{Y}_{1:\infty} )$ and the random permutation matrix is itself independent of $ (\bsY,\hat{Y}_{1:\infty} )$ given $\hat{M}=m$, the law of $ (\bsY,\mathbf{Z} )$ conditioned on $\hat{M}=m$ satisfies the assumptions of \cref{lem:multpermute}. Thus, by \cref{lem:multpermute}, $\mathbb{E}[ |f(\bsY)-g(\mathbf{Z}) | |\hat{M}=m ] = 0$ implies $f(\bsY)=c_{m}$ almost surely for some constant $c_{m}$.
(It is clear that $c_{m}$ is independent of $m$.)
\end{proof}
The next results extends \cref{cor:multperm_clutter} to the setting where a sequence of vectors $ (\bsY_{1},\ldots,\bsY_{m} )$, with $\bsY_{i}\in\mathbb{Y}^{K}$, is observed indirectly through the sequence of vectors $ (\mathbf{Z}_{1},\ldots,\mathbf{Z}_{m} )$ where each $\mathbf{Z}_{i}$ is generated as in \cref{cor:multperm_clutter} by augmenting $\bsY_{i}$ with clutter and then randomly permuting it. 

\begin{corollary}
\label{cor:multperm_vec}
Let $(\bsY_{1},\ldots,\bsY_{m})$ be a sequence of random vectors with $\bsY_{i}\in\mathbb{Y}^{K}$. Let $ (\mathbf{Z}_{1},\ldots,\mathbf{Z}_{m} )$ be conditionally independent given $ (\bsY_{1},\ldots,\bsY_{m} )$, i.e.\ $\mathbb{P}_{\mathbf{Z}_{1:m}\vert\bsY_{1:m}}=\mathbb{P}_{\mathbf{Z}_{1}\vert\bsY_{1}}\ldots\mathbb{P}_{\mathbf{Z}_{m}\vert\bsY_{m}}$ and each $\mathbb{P}_{\mathbf{Z}_{i}\vert\bsY_{i}}$ is defined as in \cref{cor:multperm_clutter}. Assume
\begin{equation}
 (P_{\psi} )^{K+m}\gg\mathbb{P}_{\bsY_{1:m}}\gg (P_{\psi} )^{K+m}\label{eq:cor_multperm_vec}
\end{equation}
If $f(\bsY_{1:m})=\mathbb{E} [ f(\bsY_{1:m})\vert\mathbf{Z}_{1:m} ]$, then $f(\bsY_{1:m})$ is a constant almost surely. In particular, (\ref{eq:cor_multperm_vec}) holds if both $ (\bsY_{1},\ldots,\bsY_{m} )$ and $P_{\psi}$ have positively valued probability densities, i.e.\ $\bsp(\bsy_{1},\ldots,\bsy_{m})>0$ and $p_{\psi}(y)>0$, w.r.t.\ the dominating measures $\nu^{K+m}$ and $\nu$ for some measure $\nu$ on $\mathbb{Y}$.
\end{corollary}

\begin{proof}
Due to the conditional independence assumption, $\mathbb{E} [ f(\bsY_{1:m})\vert\mathbf{Z}_{1},\bsY_{2:m} ] =\mathbb{E} [ f(\bsY_{1:m})\vert\mathbf{Z}_{1},\bsY_{2:m},\mathbf{Z}_{2:m} ]$ and since $f(\bsY_{1:m})=\mathbb{E} [ f(\bsY_{1:m})\vert\mathbf{Z}_{1:m} ]$, we have $f(\bsY_{1:m})=\mathbb{E} [ f(\bsY_{1:m})\vert\mathbf{Z}_{1},\bsY_{2:m} ]$ or 
\[
f(\bsY_{1},\bsy_{2:m})=\mathbb{E}[ f(\bsY_{1:m})\vert\mathbf{Z}_{1},\bsy_{2:m} ] \qquad\mathbb{P}_{\bsY_{1},\mathbf{Z}_{1}\vert\bsy_{2:m}}-\textrm{a.s.}
\]
for $\mathbb{P}_{\bsY_{2:m}}$ almost all $\bsy_{2:m}$. Invoking \cref{cor:multperm_clutter} we thus have $f(\bsY_{1},\bsY_{2:m})=g_{1}(\bsY_{2:m})$ a.s.\ for some function $g_{1}.$ Proceeding in this way we can similarly show 
\[
f(\bsY_{1:i-1},\bsY_{i},\bsY_{i+1:m})=g_{i}(\bsY_{1:i-1},\bsY_{i+1:m})
\]
for $i=1,\ldots,m$. We can now invoke assumption (\ref{eq:cor_multperm_vec}) to show that $f(\bsY_{1:m})$ must be a constant almost surely in a manner similar to the proof of \cref{lem:multpermute}; details omitted. The final statement of the corollary is straightforward.
\end{proof}

\subsubsection{Supporting results for the proof of \cref{lem:informationLoss} for deletion but no permutation uncertainty, i.e. $(\alpha=1, \beta>0)$}

The following result is an intermediate result intended to convey the main idea of the analysis while the actual full blown version is contained in \cref{lem:multidel} below.
\begin{lemma}
\label{lem:singledelete}Let $\bsY=(Y_{1},Y_{2})$, $D\in\{1,2\}$, $\bsY$ and $D$ are independent and $0<\mathbb{P}(D=1)<1$. Assume for any measurable set $C$ that $\mathbb{P} (Y_{1}\in C\vert Y_{2}\in C )<1$ whenever $0<\mathbb{P} (Y_{1}\in C )<1$. Then $f(\bsY)=\mathbb{E} [ f(\bsY)\vert\bsY_{D} ]$ implies $f(\bsY)=c$ almost surely for some constant $c$.
\end{lemma}

\begin{remark}
The random variables $(Y_{1},Y_{2})$ can always be relabelled if the assumption is met by $\mathbb{P} (Y_{2}\in\cdot\vert Y_{1}\in\cdot )$. The assumption is to be interpreted to mean that if $Y_{1}$ is not concentrated in $C$ then the event $Y_{2}\in C$ does not imply $Y_{1}\in C$. This assures some independence between the random variables.
\end{remark}
\begin{remark}
The assumption is satisfied without the need of densities. For example the assumption is satisfied when $Y_{1}$ and $Y_{2}$ are independent. If $\bsY=(Y_{1},Y_{2})$ has a density, then the assumption is satisfied if $p(y_{1},y_{2})>0$. Indeed in this case $\mathbb{P} (Y_{1}\in C^{c},Y_{2}\in C )>0$ for all $C\subset\text{\ensuremath{\mathbb{Y}}}.$
\end{remark}

\begin{proof}[Proof of \cref{lem:singledelete}] The conditional expectation can always be expressed as follows: there exists a function $g$ such that $\mathbb{E} [ f(\bsY)\vert\bsY_{D} ] = g(\bsY_{D})$ or $g(\bsY_{D})$ is a version of the conditional expectation. Thus we have 
\begin{alignat*}{1}
0 & =\mathbb{E} [ |f(\bsY)-g(\bsY_{D}) | ] \\
 & =\mathbb{P}(D=1)\mathbb{E}[  |f(\bsY)-g(Y_{1}) | ] + \mathbb{P}(D=2)\mathbb{E} [ |f(\bsY)-g(Y_{2}) | ] 
\end{alignat*}
which follows from the independence of $D$ and $\bsY$. Since $0<\mathbb{P}(D=1)<1$, both expectation terms must be zero and hence $g(Y_{1})=g(Y_{2})$ almost surely: 
\[
\mathbb{E}[  |g(Y_{1})-g(Y_{2}) | ] \leq\mathbb{E} [  |f(\bsY)-g(Y_{1}) | ] +\mathbb{E} [  |f(\bsY)-g(Y_{2}) | ] =0.
\]
For any measurable set $B$, let $A=g^{-1}(B)$. Since $g(Y_{1})=g(Y_{2})$ almost surely, we have $\mathbb{P}(g(Y_{1})\in B)=\mathbb{P}(g(Y_{1})\in B,g(Y_{2})\in B)=\mathbb{P}(g(Y_{2})\in B)$ or that 
\begin{equation}
\mathbb{P}(Y_{1}\in A)=\mathbb{P}(Y_{1}\in A,Y_{2}\in A)=\mathbb{P}(Y_{2}\in A).\label{eq:singledel_eq1}
\end{equation}
We can now show that $g(Y_{1})=c$, for some constant $c$, almost surely. (Assuming the converse is true we will arrive at a contradiction.) If $g(Y_{1})=c$ does not hold almost surely then there exist a set $B$ such that $0<\mathbb{P}(g(Y_{1})\in B)<1$ and let $A=g^{-1}(B)$ and thus $0<\mathbb{P}(Y_{1}\in A)<1$. But (\ref{eq:singledel_eq1}) implies $\mathbb{P}(Y_{1}\in A\vert Y_{2}\in A)=1$ which violates the assumption.
\end{proof}

\begin{lemma}[Multiple deletions for $K>2$.]
\label{lem:multidel} Let $\bsY=(Y_{1},\ldots,Y_{K})$ be a random vector, $D\subseteq\{1,\ldots,K\}$ and $\bsY_{D}$ denote the thinned version where components not in $D$ have been removed. Assume $0<\mathbb{P}(D=\sigma)<1$ for all subsets $\sigma\subset \{ 1,\ldots,K \}$ such that $ |\sigma |=K-1$. Furthermore, assume the following:
\begin{itemize}
\item For each $i,j\in \{ 1,\ldots,K \} $, $i\neq j$, let $Z\subseteq Y_{1:K\setminus \{ i,j \} }$. If $\mathbb{P} ( (Y_{i},Z )\in A )<1$ and $\mathbb{I}_{A}(Y_{i},Z)\neq\mathbb{E} [ \mathbb{I}_{A}(Y_{i},Z)\vert Z ]$ then $\mathbb{P} ( (Y_{i},Z )\in A\vert (Y_{j},Z )\in A )<1$. (Here $f\neq g$ means $\mathbb{P} (f\neq g )>0$.) 
\end{itemize}
Then $f(\bsY)=\mathbb{E} [ f(\bsY)\vert\bsY_{D} ]$ implies $f(\bsY)=c$ almost surely for some constant $c$.
\end{lemma}

\begin{remark}
\label{rem:multidel-rectangles} The assumption of \cref{lem:multidel}
is more easily understood as follows: for each $i,j\in \{ 1,\ldots,K \} $, $i\neq j$, let $Z\subseteq Y_{1:K\setminus \{ 1,j \} }$.
If $\mathbb{P} (Y_{i}\in A\vert Z\in B )<1$ then $\mathbb{P} (Y_{i}\in A\vert Z\in B,Y_{j}\in A )<1$

Unfortunately stating the assumption via (measurable) rectangle sets is not sufficient as the product sigma algebra is much richer than just measurable rectangles. In fact, if the set $A$ of $\mathbb{I}_{A}(Y_{i},Z)$ in the assumption of \cref{lem:multidel} is the measurable rectangle $\mathbb{I}_{A}(Y_{i})\mathbb{I}_{B}(Z)$ then the assumption reduces to that in this remark.
\end{remark}

\begin{remark}
\label{rem:strongerHypMultiDel}
The main assumption of \cref{lem:multidel} is satisfied if $\nu_{1}\times\cdots\times\nu_{K}\ll\mathbb{P}_{\bsY}\ll\nu_{1}\times\cdots\times\nu_{K}$ where $\nu_{i}$ are probability measures, i.e. $\mathbb{P}_{\bsY}$ and $\nu_{1}\times\cdots\times\nu_{K}$ are mutually absolutely continuous. (The proof of this fact uses \cref{lem:multidel_sublem1} and the fact that trace of a set is measurable; this is omitted.) 
\end{remark}

\begin{proof}[Proof of Remark \ref{rem:multidel-rectangles}] The condition $\mathbb{I}_{A}(Y_{i},Z)\neq\mathbb{E} [ \mathbb{I}_{A}(Y_{i},Z)\vert Z ]$
translates to
\eqns{
\mathbb{I}_{A}(Y_{i})\mathbb{I}_{B}(Z)\neq\mathbb{E}[ \mathbb{I}_{A}(Y_{i})\mathbb{I}_{B}(Z)\vert Z ]
}
and from \cref{lem:multidel_sublem1} this implies that for all measurable sets $C$ in the range of $Z$ 
\[
\mathbb{E}[ \mathbb{I}_{A}(Y_{i})\mathbb{I}_{B}(Z) ] \neq\mathbb{E}[ \mathbb{I}_{C}(Z) ].
\]
Taking $C=B$ this then implies $\mathbb{P} (Y_{i}\in A,Z\in B )<\mathbb{P} (Z\in B )$
and so 
\[
\mathbb{P} (Y_{i}\in A\vert Z\in B )<1.
\]
Clearly the probability $\mathbb{P} ( (Y_{i},Z )\in A\vert (Y_{j},Z )\in A )$ of the assumption of \cref{lem:multidel} for the rectangle is $\mathbb{P} (Y_{i}\in A\vert Y_{j}\in A,Z\in B ).$ 
\end{proof}

\begin{proof}[Proof of \cref{lem:multidel}]
The random variable $\bsY_{D}$ belongs to $\bbY^{\times}$, that is to the disjoint union $\cup_{k=0}^{K}\mathbb{Y}^{k}$ with $\mathbb{Y}^{0}\equiv\emptyset.$ Thus we can write
\begin{alignat*}{1}
0=\mathbb{E}[ |f(\bsY)-g(\bsY_{D}) | ] & =\sum_{i=0}^{K}\sum_{\sigma:|\sigma|=i}\mathbb{E}[  |f(\bsY)-g_{i}(\bsY_{\sigma}) | \mathbb{I}_{[D=\sigma]} ] \\
 & =\sum_{i=0}^{K}\sum_{\sigma:|\sigma|=i}\mathbb{E}[  |f(\bsY)-g_{i}(\bsY_{\sigma}) | ] \mathbb{P}(D=\sigma)
\end{alignat*}
where $g_{0}$ is a constant, $g_{i}:\mathbb{Y}^{i} \to \mathbb{R}$ are measurable functions and independence of $D$ and $\bsY$ has been invoked. If $\mathbb{P}(D=\emptyset)>0$ then it is trivial since this implies $\mathbb{E}[ |f(\bsY)-g_{0} | ] = 0$. So assume $\mathbb{P}(D=\emptyset)=0$.

Having assumed $0<\mathbb{P}(D=\sigma)<1$ for all subsets $\sigma\subset \{ 1,\ldots,K \}$ such that $ |\sigma |=K-1$, we focus on these terms only:
\begin{alignat*}{1}
\sum_{\sigma:|\sigma|=K-1}\mathbb{E}[  |f(\bsY)-g_{K-1}(\bsY_{\sigma}) |] \mathbb{P}(D=\sigma)
\end{alignat*}
which also implies 
\begin{equation}
g_{K-1}(\bsY_{\sigma})=g_{K-1}(\bsY_{\sigma'})\quad\text{or}\quad\mathbb{I}_{A}(\bsY_{\sigma})=\mathbb{I}_{A}(\bsY_{\sigma'})\quad\text{almost surely}\qquad\label{eq:multidel_eq1}
\end{equation}
for all $\sigma,\sigma'$ and $A=g_{K-1}^{-1}(B)$ for a measurable set $B$ in $\mathbb{R}$. For example, when $\sigma=(1,3,\ldots,K)$, $\sigma=(2,3,\ldots,K)$ and $Z= (Y_{3},\ldots,Y_{K} )$, we get 
\[
\mathbb{P} ( (Y_{1},Z )\in A )=\mathbb{P} ( (Y_{1},Z )\in A, (Y_{2},Z )\in A )=\mathbb{P} ( (Y_{2},Z )\in A ).
\]

Henceforth we refer to $g_{K-1}$ simply as $g$. We need to show that $g(\bsY_{\sigma})=c$, for some constant $c$, almost surely.  If this is not the case then there exists subsets of variables $Y_{i}\in\bsY_{\sigma}$, $Z\subset\bsY_{\sigma}$ and $Y_{i}\notin Z$ (recall $\sigma\subset \{ 1,\ldots,K \}$ with $ |\sigma |=K-1$) such that 
\begin{equation}
g(\bsY_{\sigma})=\mathbb{E}[ g(\bsY_{\sigma}) |Y_{i},Z ] \quad\text{a.s. and}\quad g(\bsY_{\sigma})\neq\mathbb{E}[ g(\bsY_{\sigma}) |Z ] \quad\text{a.s.}\qquad\label{eq:multidel_eq2}
\end{equation}
The interpretation is that $g(\bsY_{\sigma})$ can potentially be a function of the reduced set of variables $(Y_{i},Z)$ (as asserted by the first equality) but it must genuinely be a function of at least
the variable $Y_{i}$. 

For clarity and simplicity assume $i=1$ and $Z= (Y_{3},\ldots,Y_{K} )$. Consider the terms in the sum due to $\sigma=(1,3,\ldots,K)$ and $\sigma=(2,3,\ldots,K)$. By \cref{lem:multidel_sublem2} below, there exists a measurable set $A=g^{-1}(B)$ such that $0<\mathbb{P} (A )<1$ and $\mathbb{I}_{A}(Y_{1},Z)\neq\mathbb{E}[ \mathbb{I}_{A}(Y_{1},Z)\vert Z ]$.  But (\ref{eq:multidel_eq1}) implies $\mathbb{P}( (Y_{1},Z )\in A\vert (Y_{2},Z )\in A)=1$ which violates the main assumption of the lemma.
\end{proof}

The following two lemmas concerns the random variables declared in the statement of \cref{lem:multidel}

\begin{lemma} \label{lem:multidel_sublem1} $\mathbb{I}_{A}(Y_{i},Z)=\mathbb{E}[ \mathbb{I}_{A}(Y_{i},Z)\vert Z ]$ a.s.\ if and only if there exits a set $C$ in $\sigma(Z)$ such that $\mathbb{I}_{A}(Y_{i},Z)=\mathbb{I}_{C}$ almost surely. 
\end{lemma}

\begin{proof}
Let $F$ denote $\mathbb{E}[ \mathbb{I}_{A}(Y_{i},Z)\vert Z ]$ and let $C= \{ \omega:\thinspace F(\omega)=1 \} $. Clearly $C\in\sigma(Z)$ and we show $\mathbb{I}_{A}(Y_{i},Z)=\mathbb{I}_{C}$ almost surely. 

The set $\mathbb{I}_{A}(Y_{i},Z)\neq F$ has measure zero and is comprised of the disjoint sets $ \{ \omega:\thinspace\mathbb{I}_{A}(Y_{i},Z)=1,F\neq1 \}$, $ \{ \omega:\thinspace\mathbb{I}_{A}(Y_{i},Z)=0,F=1 \}$ and $ \{ \omega:\thinspace\mathbb{I}_{A}(Y_{i},Z)=0,F\neq0,F\neq1 \}$.  Thus the result follows since $\mathbb{I}_{C}$ and $\mathbb{I}_{A}(Y_{i},Z)$ differ precisely on the first two sets. (The reverse implication is obvious.)
\end{proof}

\begin{lemma}
\label{lem:multidel_sublem2} If $g(Y_{i},Z)\neq\mathbb{E}[ g(Y_{i},Z)\vert Z ]$ then there exists a measurable function $\mathbb{I}_{A}(Y_{i},Z)$ such that $0<\mathbb{P} ((Y_{i},Z)\in A )<1$ and $\mathbb{I}_{A}(Y_{i},Z)\neq\mathbb{E}[ \mathbb{I}_{A}(Y_{i},Z)\vert Z ]$.
\end{lemma}

\begin{proof}
The proof is complete once we show that $g(Y_{i},Z)\neq\mathbb{E}[ g(Y_{i},Z)\vert Z ]$ implies there exists a set $A'\in\sigma(g)$ such that $\mathbb{I}_{A'}\neq\mathbb{E}[ \mathbb{I}_{A'}\vert Z ]$, which in turn implies $0<\mathbb{P} (A' )<1$. (Here we write $\sigma(g)$, rather than $\sigma (g(Y_{i},Z) )$, for the smallest $\sigma$-algebra with respect to which the random variable $g(Y_{i},Z)$ is measurable.) Note that for every $A'\in\sigma(g)$ there exists a measurable set $B$ in $\mathbb{R}$ such that $A'= \{ \omega:(Y_{i}(\omega),Z(\omega))\in g^{-1}(B) \} $.

Assume that $\mathbb{I}_{A'}=\mathbb{E}[ \mathbb{I}_{A'}\vert Z ]$ almost surely for all $A'\in\sigma(g)$. Then by \cref{lem:multidel_sublem1}, for each $A'\in\sigma(g)$ there is a set $C$ in $\sigma(Z)$ such that $\mathbb{I}_{A'}=\mathbb{I}_{C}$ almost surely. This will violate $g(Y_{i},Z)\neq\mathbb{E}[ g(Y_{i},Z)\vert Z ]$ if $g(Y_{i},Z)=\mathbb{I}_{A'}$ and hence also if $g$ was a simple function; if we replace each indicator function (measurable set) of the simple function with its almost sure counterpart from $\sigma(Z)$ then it becomes $\sigma(Z)$ measurable. In the general case, the result may be established since $g$ may be approximated by a sequence of simple functions $ \{ g_{n} \}$ tending to $g$ and each simple function $g_{n}$ itself has a $\sigma(Z)$ measurable version, say $\overline{g}_{n}$, and $\overline{g}_{n}=g_{n}$. Thus $\mathbb{E}[ g_{n}\vert Z] =\mathbb{E}[ \overline{g}_{n}\vert Z] =g_{n}$. Letting $\overline{g}=\mathbb{E}[ g(Y_{i},Z)\vert Z ]$, we see that $\mathbb{E}[|g-\overline{g} |] = \lim\mathbb{E}\big[ |g_{n}-\mathbb{E}[g_{n}\vert Z] |\big]=0.$ 
\end{proof}

\section{Proof of \texorpdfstring{\cref{thm:informationLossPd}}{Theorem~\ref{thm:informationLossPd}}}
\label{proof:thm:informationLossPd}
The case where $K^*=1$ is first considered so that the number of observations $M_t$ at time $t$ can only be equal to zero or one. The joint probability of the observations and states becomes
\eqns{
\bar\bsp_{\bstheta}(\bsy_{1:n}, x_{0:n}) = \pi_{\theta}(x_0) \prod_{t=1}^n \big[ (1-p_{\D})^{1-M_t}(p_{\D}g_{\theta}(\bsy_t \given x_t))^{M_t} f_{\theta}(x_t \given x_{t-1}) \big],
}
where $\bsy_t$ is the empty sequence when $M_t = 0$. The size of $\bsy_t$ at any time $t$ can be made explicit in this expression for the sake of clarity as follows
\eqns{
\bar\bsp_{\bstheta}(\bsy_{1:n}, x_{0:n}, m_{1:n}) = \pi_{\theta}(x_0) \prod_{t=1}^n \big[ (1-p_{\D})^{1-m_t}(p_{\D}g_{\theta}(\bsy_t \given x_t))^{m_t} f_{\theta}(x_t \given x_{t-1}) \big].
}

Let $Y_t^{\epsilon}$ be a noisy version of the original observation $Y_t$ for any $t \geq 1$ so that the HMM $(X_t,Y^{\epsilon}_t)_t$ is equal in law to the HMM $(X_t,Y_t + \epsilon Z_t)_t$ where $(Z_t)_t$ is an i.i.d.\ sequence of random variables which common law is the uniform distribution over the ball of radius $1$ and centre $0$. A switching process $(s_t)_t$ is also introduced as follows: $s_t = 1$ when the target is detected and $s_t = 0$ otherwise. In order to study the Fisher information more easily, we introduce an alternative observation model where a detection failure at time $t$ is replaced by an observation $Y^{\epsilon}_t$ from the target. The law of this observation model is
\eqnsa{
\bar{p}_{\theta}^{\epsilon}(\tilde{y}_{1:n}, x_{0:n}, s_{1:n}) & = \pi_{\theta}(x_0) \prod_{t=1}^n \big[ p_{\D}^{s_t} (1-p_{\D})^{1-s_t} g_{\theta}(\tilde{y}_t \given x_t) f_{\theta}(x_t \given x_{t-1}) \big] \\
& = \pi_{\theta}(x_0) \prod_{t=1}^n \Big[ [p_{\D} g_{\theta}(y_t \given x_t)]^{s_t} [(1-p_{\D}) g_{\theta}(y^{\epsilon}_t \given x_t)]^{1-s_t} f_{\theta}(x_t \given x_{t-1}) \Big]
}
where $\tilde{y}_t = y_t$ if $s_t = 1$ and $\tilde{y}_t = y^{\epsilon}_t$ if $s_t = 0$. The quantity of interest is
\eqnsml{
\bar{p}_{\theta}^{\epsilon}(\tilde{y}_0, s_0 \given y_{-\infty:-1},y^{\epsilon}_{1:\infty}) = \\
[p_{\D} g_{\theta}(y_0 \given x_0)]^{s_0} [(1-p_{\D}) g_{\theta}(y^{\epsilon}_0 \given x_0)]^{1-s_0} \bar{p}_{\theta}(x_0 \given y_{-\infty:-1},y^{\epsilon}_{1:\infty}),
}
which we compare with the full-detection case
\eqns{
\bar{p}_{\theta}^{\epsilon}(y_0 \given y_{-\infty:-1},y^{\epsilon}_{1:\infty}) = g_{\theta}(y_0 \given x_0) \bar{p}_{\theta}(x_0 \given y_{-\infty:-1},y^{\epsilon}_{1:\infty}).
}
To justify the equivalence of the two observation model for the considered purpose, we can verify that the score $\nabla_{\theta} \log \bar{\bsp}_{\bstheta}(\bsy_0, m_0 \given y_{-\infty:-1})$ is equal to the score $\nabla_{\theta} \log \bar{p}_{\theta}^{\epsilon}(\tilde{y}_0, s_0 \given y_{-\infty:-1},y^{\epsilon}_{1:\infty})$ when $\epsilon \to \infty$. With the required modifications and after \cite[Theorem~5]{Dean2014}, it follows that the loss of information $I^{\epsilon}_{\loss}(\theta^*)$ when replacing the original observations by the $\epsilon$-perturbed ones can be expressed as
\eqnsml{
I^{\epsilon}_{\loss}(\theta^*) = \bar\bbE_{\theta^*} \Big[ \nabla_{\theta} \log \bar{p}_{\theta^*}(\bsY_0 \given \bsY_{-\infty:-1}, \bsY^{\epsilon}_{1:\infty}) \cdot \nabla_{\theta} \log \bar{p}_{\theta^*}(\bsY_0 \given \bsY_{-\infty:-1}, \bsY^{\epsilon}_{1:\infty})^{\tr} \Big] \\
- p_{\D} \bar\bbE_{\theta^*} \Big[ \nabla_{\theta} \log \bar{p}_{\theta^*}(\bsY_0 \given \bsY_{-\infty:-1}, \bsY^{\epsilon}_{1:\infty}) \cdot \nabla_{\theta} \log \bar{p}_{\theta^*}(\bsY_0 \given \bsY_{-\infty:-1}, \bsY^{\epsilon}_{1:\infty})^{\tr} \Big] \\
- (1-p_{\D}) \bar\bbE_{\theta^*} \Big[ \nabla_{\theta} \log \bar{p}_{\theta^*}(\bsY^{\epsilon}_0 \given \bsY_{-\infty:-1}, \bsY^{\epsilon}_{1:\infty}) \cdot \nabla_{\theta} \log \bar{p}_{\theta^*}(\bsY^{\epsilon}_0 \given \bsY_{-\infty:-1}, \bsY^{\epsilon}_{1:\infty})^{\tr} \Big].
}
Considering the limit $\epsilon \to \infty$, it follows that
\eqnsa{
I_{\loss}(\theta^*) & \defeq \lim_{\epsilon \to \infty} I^{\epsilon}_{\loss}(\theta^*) \\
& = (1-p_{\D}) \bar\bbE_{\theta^*} \Big[ \nabla_{\theta} \log \bar\bsp_{\theta^*}(\bsY_0 \given \bsY_{-\infty:-1}) \cdot \nabla_{\theta} \log \bar\bsp_{\theta^*}(\bsY_0 \given \bsY_{-\infty:-1})^{\tr} \Big].
}
In the multi-target case, it simply holds that the information loss is equal to $(1-p^*_{\D}) K^* I(\theta^*)$ since targets' detection are independent when the data association is known, which terminates the proof of the proposition.

\bibliographystyle{siamplain}
\bibliography{Identifiability}

\end{document}